\documentclass[12pt,a4paper]{article}
\usepackage{amsmath,amssymb,amsthm,amsfonts,amscd,euscript,verbatim, t1enc, newlfont, xypic}
\usepackage{hyperref}

 \theoremstyle{plain}
\newtheorem{theo}{Theorem}[subsection]

\newtheorem{pr}[theo]{Proposition}

 \newtheorem{lem}[theo]{Lemma}
 \newtheorem{coro}[theo]{Corollary}
  \newtheorem{conj}[theo]{Conjecture}
\theoremstyle{remark}
\newtheorem{rema}[theo]{Remark}

\theoremstyle{definition}
\newtheorem{defi}[theo]{Definition}
\newtheorem*{notat}{Notation}

 \newcommand\lan{\langle}
\newcommand\ra{\rangle}

\newcommand\ob{^{-1}}

\newcommand\dmge{DM^{eff}_{gm}{}}

\newcommand\dmgm{DM_{gm}}

\newcommand\obj{Obj}

\newcommand\id{id}
\newcommand\cu{\underline{C}}
\newcommand\du{{\underline{D}}}
\newcommand\eu{{\underline{E}}}
\newcommand\au{\underline{A}}
\newcommand\bu{\underline{B}}

\newcommand\aum{\underline{A}_m}
\newcommand\mmi{MM_i}
\newcommand\mmm{MM_m}

\newcommand\mm{MM}
\newcommand\spe{{\operatorname{Spec}\,}}

\newcommand\hrt{{\underline{Ht}}}
\newcommand\hrtl{{\underline{Ht}_l}} 
\newcommand\htlz{H^{t_l}_0}
\newcommand\hw{{\underline{Hw}}}

\newcommand\wmmli{W_{\le i,MM}}

\newcommand\z{{\mathbb{Z}}}

\newcommand\zlz{\z/l\z}

\newcommand\ql{{\mathbb{Q}_l}}
\newcommand\qlp{{\mathbb{Q}_{l'}}}
\newcommand\qp{{\mathbb{Q}_p}}

\newcommand\zol{{\mathbb{Z}[\frac{1}{l}]}}
\newcommand\szol{{\spe \mathbb{Z}[\frac{1}{l}]}}
\newcommand\zollp{\spe\mathbb{Z}[\frac{1}{l},\frac{1}{l'}]}

\newcommand\q{{\mathbb{Q}}}
\newcommand\sq{{\operatorname{Spec}\mathbb{Q}}}

\newcommand\p{\mathbb{P}}

\newcommand\ns{\{0\}}

\DeclareMathOperator\prli{\varprojlim}
\DeclareMathOperator\inli{\varinjlim}

\DeclareMathOperator\gal{\operatorname{Gal}}

\newcommand\chow{Chow}

\newcommand\chowe{Chow^{eff}}

\newcommand\dshsl{D^b_cSh^{et}(S[1/l],\ql)}
\newcommand\dshkl{D^b_cSh^{et}(K,\ql)} 

\newcommand\dhsl{D^b_cSh^{et}(S,\ql)}

\newcommand\dhslw{{\tilde D^b_cSh^{et}(S,\ql)}}
\newcommand\dhsli{D^b_cSh^{et}(S_i,\ql)}
\newcommand\dhuli{D^b_cSh^{et}(U_i,\ql)}
\newcommand\dhzl{D^b_cSh^{et}(Z,\ql)}

\newcommand\dhupl{D^b_cSh^{et}(U',\ql)}
\newcommand\dhuppl{D^b_c Sh^{et}(U'',\ql)}

\newcommand\dhklp{D^b_c Sh^{et}(K,\qlp)}

\newcommand\dshslp{D^b_c Sh^{et}(S[1/l'],{\mathbb{Q}_{l'}})}
\newcommand{\tlp}{{t_{l'}}}
\newcommand{\tllp}{{t_{l,l'}}}

\newcommand\dshs{D^b_c Sh^{et}(S,\ql)}
\newcommand\shs{Sh_{per}^{et}(S,\ql)}
\newcommand\shm{Sh_{per}^{et}(-,\ql)}
\newcommand\dshl{D^b_c Sh^{et}(-,\ql)}

\newcommand\zl{{\mathbb{Z}_l}}
\newcommand\dmhsz{DM_h(S,\z)} 
\newcommand\dbcszl{D^b_c (S,\zl)}

 \DeclareMathOperator\ke{\operatorname{Ker}}
 
\DeclareMathOperator\imm{\operatorname{Im}}
\DeclareMathOperator\co{\operatorname{Cone}}

\DeclareMathOperator\cha{\operatorname{char}}

\newcommand\hetl{{\mathcal{H}}^{et}_{\ql}{}}
\newcommand\hetlw{{\tilde{\mathcal{H}}}^{et}_{\ql}{}}

\newcommand\hetlz{H^{et}_{\ql,0}{}}

\newcommand\hetlwz{{\tilde H}^{et}_{\ql,0}{}}
\newcommand\hetlm{H^{et}_{\ql,m}{}}
\newcommand\hetlmw{{\tilde H^{et}_{\ql,m}}{}}

\newcommand\hetln{H^{et}_{\ql,n}{}}
\newcommand\hetlj{H^{et}_{\ql,j}{}}

\newcommand\hetlp{{\mathcal{H}}^{et}_{{\mathbb{Q}_{l'}}}{}}
\newcommand\hetp{{\mathcal{H}}^{et}_{\qp}{}}

\newcommand\dms{DM(S)}
\newcommand\dmcs{DM_c(S)}

\newcommand\dmcu{DM_c(U)}
\newcommand\dmck{DM_c(K)}

\newcommand\dmcx{DM_c(X)}
\newcommand\dm{DM}
\newcommand\dmc{DM_c}

\newcommand\dmcy{DM_c(Y)}
\newcommand\chows{Chow(S)}
\newcommand\hwchow{{\underline{Hw}}_{\chow}}

\newcommand\wchow{{w_{Chow}}}

\newcommand\sss{{\mathcal{S}}}

\newcommand\fp{\mathbb{F}_p}
\newcommand\fl{\mathbb{F}_l}

\newcommand\sfp{\operatorname{Spec}\mathbb{F}_p}

\newcommand\dbm{D^b_m}

\begin{document}

 \title{Mixed motivic sheaves (and weights for them) exist if 'ordinary' mixed motives do}
 \author{Mikhail V. Bondarko
   \thanks{ 
 The work is supported by RFBR
(grants no. 12-01-33057 and 14-01-00393). 
} }
 \maketitle
\begin{abstract}
The goal of this paper is to prove: if certain 'standard' conjectures on motives over algebraically closed fields hold, then over any 'reasonable' scheme $S$ there exists a {\it motivic} $t$-structure for the category $\dmcs$ of relative Voevodsky's motives (being more precise, for the Beilinson motives described by Cisinski and Deglise). If $S$ is 
of finite type over a field, then the  heart of this $t$-structure (the category of {\it mixed motivic sheaves} over $S$) is endowed with a {\it weight filtration} with semisimple factors. 
We also prove a certain 'motivic decomposition theorem' (assuming the conjectures mentioned) and characterize semisimple 
motivic sheaves 
over 
$S$ in terms of those over its residue fields. 

Our main tool is the theory of {\it weight structures}. We actually prove somewhat more than the existence of a weight filtration for mixed motivic sheaves: we prove that the motivic $t$-structure is {\it transversal} to the Chow weight structure for $\dmcs$ (that was introduced previously 
 by D. H\'ebert and the author). We also deduce several properties of mixed motivic sheaves from this fact. Our reasoning relies on the degeneration of {\it Chow-weight spectral sequences}  for 'perverse \'etale homology' (that we prove unconditionally); this statement also yields the existence of  the {\it Chow-weight filtration} for such (co)homology that is strictly restricted by ('motivic') morphisms.

\end{abstract}

\tableofcontents

\section*{Introduction}

The famous conjectures of Beilinson (see \S5.10A in \cite{beilh}) predict the existence of an abelian category
$MM(S)$ of {\it mixed motivic sheaves} over a (more or less, arbitrary) scheme $S$. This category should be endowed with a so-called {\it weight filtration} whose factors are semisimple; it should possess an exact realization to the category of perverse ($\ql$-) \'etale sheaves. 
The goal of this paper is to deduce these conjectures
from certain 'standard' conjectures on motives over algebraically closed fields.

Now we explain this in more detail.  It is  widely believed that 
$MM(S)$ should be the heart of a certain (motivic) $t$-structure for some triangulated category 
 of (Voevodsky's) motives over 
$S$. 
In this paper we treat this question for the category $\dmcs$ 
 of constructible Beilinson motives (as described in \cite{degcis}) over a 
 ({\it very reasonable})
 base scheme $S$, and prove that a ('nice') motivic $t$-structure exists for it if it exists for Voevodsky's motives over algebraically closed fields.
Recall here: already the latter assumption requires certain very hard 'standard' conjectures (especially for positive characteristic fields; see \S\ref{sconj} below and \S2 of \cite{ha3} for a 
discussion of those), yet it is nice to know that 
passing to relative motives in this matter conceals no additional difficulties. 
Note also:  the paper 
\cite{haco} relies on the same  conjectures that we need for our main results, whereas in ibid. only the properties of (certain) 'pure' relative motivic sheaves (and only for $S$ being a variety over a characteristic $0$ field) are established. 
In particular, we  prove a certain motivic version of the Decomposition Theorem for perverse sheaves (see \S\ref{sdecomp}) that is much stronger than the corresponding result of \cite{haco}. 
We also characterize 
 simple 
 mixed motivic sheaves (those are certainly 'pure') in terms of those over the residue fields of $S$. 
Certainly, the results of \cite{degcis} are crucial for our success here.

Now we describe our central results in more detail, and also mention the main prerequisites for their proofs.

Our first principal result is the following one. Suppose that for some fixed prime $l$ and for any universal domain $K$ (of characteristic distinct from $l$) there exists a $t$-structure $t_{MM}$ for the category $\dmgm(K)$ of Voevodsky's motives over $K$ that is strictly compatible with the (canonical) $t$-structure for $\ql$-adic \'etale sheaves (via 
 \'etale 
homology; we call 
the heart of $t_{MM}$ 
the category of mixed motives over $K$). Then 
$t_{MM}$ also exists for motives over any 'reasonable' (see below) $\spe\zol$-scheme $S$. So, one may say that a certain $MM(S)$ exists in this case. The proof is quite easy (given the properties of Beilinson motives established in \cite{degcis}); 
we just apply a simple gluing argument. 
Actually, it is not necessary to fix $l$ here: if for each $K$ of characteristic $p$ there exists a (motivic) $t$-structure for $\dmgm(K)$ that is strictly compatible (as above) with $\qlp$-adic \'etale (co)homology for any $l'\neq p$, then the motivic $t$-structure exists over any (reasonable) $S$ (and it does not depend on $l$). 

The second central result seems to be more interesting; its proof is more  complicated. We verify that  certain  'weights' exist for mixed motives over any 
 {\it very reasonable} (see Definition \ref{dmor}(4) below) scheme $S$; if $\cha S=0$ or if "the weights are nice" over a universal $K$ such that $\cha K=\cha S$, then these weights are 'nice' over $S$ also.
This sentence requires a considerable amount of explanation, and we give it here. 

The 'classical' approach for constructing weights for motives (originating from Beilinson) 
was to define a filtration for motives that would split Chow motives into their components corresponding to single (co)homology groups (i.e would yield the so-called Chow-Kunneth decompositions). Since  the existence of Chow-Kunneth decompositions is very much conjectural, 
 it is no wonder  that this approach 
 has not yielded any significant (general) results up to this moment (to the knowledge of the author).  

An alternative method for defining (certain) weights for motives
was proposed and successfully implemented in \cite{bws}. To this
end {\it weight structures} for triangulated categories were defined. This notion is a natural
important counterpart of $t$-structures; somewhat similarly to
$t$-structures, weight structures for a triangulated $\cu$ are
defined in terms of $\cu_{w\le 0},\cu_{w\ge 0}\subset \obj\cu$. For the Chow weight structure $\wchow$ for $\dmcs$ 
its {\it heart} $\dmcs_{\wchow\le 0}\cap\dmcs_{\wchow\ge 0}$  consists of Chow motives (over $S$; those are 'ordinary' Chow motives if $S$ is the spectrum of a perfect field); we avoid Chow-Kunneth decompositions this way.
 $\wchow$  allows  to define certain
{\it (Chow)-weight filtrations} and {\it (Chow)-weight spectral sequences} for
any (co)homology of motives; for singular and \'etale cohomology
those are isomorphic to the 'classical' ones. 
$\wchow$ for  
$\dmcs$ was introduced in \cite{hebpo} and \cite{brelmot}; it is closely
related with the weights for  mixed complexes of sheaves (as introduced in
\S5.1.8 of \cite{bbd}; see \S\S3.4--3.6 
 of \cite{brelmot}) and of mixed Hodge complexes and modules (see \S2.3 of \cite{btrans}). All of these results are unconditional.
 
In \cite{btrans} and (especially) in the current paper we demonstrate that the Chow weight structure is also useful for the study of motivic conjectures. In particular, we prove (using the results of \S3 of \cite{btrans}): if $t_{MM}$ exists over a scheme $S$ that is of finite type over a field, then Chow-weight spectral sequences yield a  weight filtration for 
$S$-motivic sheaves;  this filtration is strictly respected by morphisms of motives.
Our argument relies on the degeneration at $E_2$ of  Chow-weight spectral sequences for the 'perverse \'etale homology'.  We prove the latter result unconditionally. It also yields the existence of the Chow-weight filtration for perverse \'etale (co)homology of motives that is strictly restricted by 'motivic' morphisms; so it could be useful for itself. The proof relies on certain new 'continuity' properties of the Chow weight structure.

Moreover, in (\S1 of) ibid. also the conjectural relation of $\wchow$ with the motivic $t$-structure was axiomatized. The corresponding notion of {\it transversal} weight and $t$-structures was introduced, and several equivalent definitions of transversality were given.
So, we actually prove: if 
over a universal domain $K$ of characteristic $p$ (that could be $0$) $t_{MM}$ exists and is transversal to $\wchow$, then the same is true for $\dmcs$ for any very reasonable $S$  of characteristic $p$. 

This 'triangulated' approach to weights (for mixed motives) has serious advantages over the (usual) 'abelian' version. First, it allows to combine the conjectural properties of mixed motives with  unconditional results on the Chow weight structure (and on Chow-weight spectral sequences). We obtain 
some 'new' properties of mixed motivic sheaves  this way; note that (by  virtue of our results) all of them follow from  'standard' motivic conjectures (cf. the discussion in \S\ref{sconj} below). Besides, we obtain a description of weights for mixed motives whose only conjectural ingredient is the existence of 
$t_{MM}$. Lastly note that the 'triangulated' 
approach  allows us to apply a certain gluing argument (that heavily relies on \S1.4 of \cite{bbd}) that does not seem to work in the context of filtered abelian categories. 

Summarizing: we prove that if a certain list of standard (motivic) conjectures over algebraically closed fields hold, then the category of mixed motivic sheaves exists over any reasonable scheme $S$; for $S$ that is also very reasonable we  obtain 'nice weights' for $S$-motivic sheaves. 
We also deduce 
(most of) the properties 
of this category that were conjectured by Beilinson and others, and prove some of their 'triangulated extensions'. Besides, we prove a certain 'motivic decomposition theorem', and calculate the Grothendieck group of mixed motivic sheaves. 

Lastly, we note that the results of the current paper (as well as the results of \cite{brelmot}) 
only rely upon a certain 'axiomatics' of Beilinson motives (i.e. on a certain list of their properties; cf. Remark \ref{rmain}(1) below). It follows that our arguments could be applied to the study of other categories satisfying similar properties. A natural candidate here would be M. Saito's Hodge modules. Yet it seems that this setting has been already thoroughly studied by Saito himself  (cf. Proposition 2.3.1(I) of \cite{btrans}); on the other hand, our methods could possibly yield certain simplifications for his arguments.

Now we list the contents of the paper. More details can be
found at the beginnings of 
sections.

\S1 is dedicated to the recollection of certain homological algebra. We recall some basics of $t$-structures. We also remind the reader  basic definitions and results on weight structures, weight filtrations and spectral sequences, as well as the notion of transversality of weight and $t$-structures (following \cite{bws} and \cite{btrans}). We also recall (mostly from \S1.4 of \cite{bbd}) several basic results on gluing of $t$-structures and weight structures.

In \S2 we recall the basic properties of $S$-motives (as defined and studied in \cite{degcis}) and the Chow weight structure for them (as introduced in \cite{hebpo} and \cite{brelmot}; we prove some new 'continuity' properties of the Chow weight structure). We also recall some properties of the perverse $t$-structure for $\ql$-\'etale sheaves, and 
study weight spectral sequences for the 'perverse \'etale homology'. Those degenerate
at $E_2$ if $S$ is a very reasonable scheme; we conjecture that they degenerate for a general (reasonable) $S$ also.

In \S3 we define the motivic $t$-structure (when it exists) as the one that is (strictly) compatible with the perverse $t$-structure for complexes of $\ql$-adic sheaves. We prove that the   motivic $t$-structure exists over $S$ if it exists over (all) universal domains. 
We also deduce some simple consequences from the 'niceness' of the motivic $t$-structure (i.e. of its transversality with $\wchow$). They enable us to prove: over a very reasonable scheme there exists a nice $t_l$  if the same is true over some universal domain of the same characteristic.

In \S4 we verify that the existence of a (nice) motivic $t$-structure and its independence from $l$ follows from a certain  list of (more or less) 'standard' motivic conjectures (over algebraically closed base fields). We also prove a  certain 'motivic Decomposition Theorem' (modulo the conjectures mentioned). 
In particular, we characterize semisimple (pure) motives over $S$ in terms of those over its residue fields. This enables us to calculate $K_0(\dmcs)$. 

The author is deeply grateful to prof. B. Conrad, prof. D.-Ch. Cisinski,   prof. M. de Cataldo,  prof.
F. Deglise,   prof. V. Guletskii, and to the users of the Mathoverflow forum for their interesting comments and important advice.  
He would also like to express his gratitude to the officers and the guests of the Max Planck Institut f\"ur Mathematik, as well as to prof. M. Levine and  the Essen University, and to prof. F. Lecomte and the Strasbourg University,  for the wonderful working conditions during the work on this paper.

\begin{notat}

$\cu$ below will always denote some triangulated category.
$t$ will always denote a bounded $t$-structure,
and $w$ will be a bounded weight structure  (the theory
of weight structures was thoroughly studied  in \cite{bws}; see also \S\ref{sws}
 below).


$D\subset \obj \cu$ will be
called extension-stable 
    if for any distinguished triangle $A\to B\to C$
in $\cu$ we have: $A,C\in D\implies
B\in D$. Note that $\cu^{t\le i}$, $\cu^{t\ge i}$, $\cu^{t=0}$ (see \S\ref{dtst}), 
 $ \cu_{w\ge i}$, and $\cu_{w\le i}$ (see \S\ref{sws}) 
 are extension-stable for any $t,w$ and any $i
 \in \z$.

For $D,E\subset \obj \cu$ we will write $D\perp E$ if $\cu(X, Y)=\ns$
 for all $X\in D,\ Y\in E$.
 For $D\subset \cu$ we will denote by $D^\perp$ the class
$$\{Y\in \obj \cu:\ X\perp Y\ \forall X\in D\}.$$
Dually, ${}^\perp{}D$ is the class
$\{Y\in \obj \cu:\ Y\perp X\ \forall X\in D\}$.

A full subcategory $B\subset \cu$ is called {\it Karoubi-closed} in $\cu$
if $B$ contains all
$\cu$-retracts of its objects
For $B\subset \cu$ we will call  the subcategory of $\cu$ whose objects are  all
retracts of  objects of $B$ (in $\cu$) the {\it Karoubi-closure} of $B$
in $\cu$.

  For a class of
objects $C_i\in\obj\cu$, $i\in I$, we will denote by $\lan C_i\ra$
the smallest strictly full triangulated subcategory containing all $C_i$. 
We will call the  Karoubi-closure of $\lan C_i\ra$ in $\cu$ the {\it triangulated category  generated by $C_i$}.


$\au$ will always be an abelian category.
We will call a covariant (resp. contravariant)
additive functor $H:\cu\to \au$ 
 {\it homological} (resp. {\it cohomological}) if
it converts distinguished triangles into long exact sequences. 

\end{notat}

\section{Preliminaries on triangulated categories, weight- and $t$-structures}\label{stria}

In \S\ref{sws} we recall some basics 
on weight structures (as developed in \cite{bws}). 

In \S\ref{dtst} we recall the definition of a $t$-structure and introduce some notation. 

In \S\ref{swss} we study weight spectral sequences (following \S2 of \cite{bws} and \S3 of \cite{btrans}), their degeneration, and weight filtrations for $\hrt$ coming from $w$.

In \S\ref{strans} we recall the notion of transversal weight and $t$-structures (as introduced in \cite{btrans}).

In \S\ref{sglu} we prove (heavily relying upon \S1.4 of \cite{bbd}) several auxiliary statements on  $t$-structures and weights in the 'gluing setting'.

\subsection{Weight structures: short reminder}\label{sws}

\begin{defi}\label{dwstr}
I A pair of subclasses $\cu_{w\le 0},\cu_{w\ge 0}\subset\obj \cu$ 
will be said to define a weight
structure $w$ for $\cu$ if 
they  satisfy the following conditions:

(i) $\cu_{w\ge 0},\cu_{w\le 0}$ are additive and Karoubi-closed in $\cu$
(i.e. contain all $\cu$-retracts of their objects).

(ii) {\bf Semi-invariance with respect to translations.}

$\cu_{w\le 0}\subset \cu_{w\le 0}[1]$, $\cu_{w\ge 0}[1]\subset
\cu_{w\ge 0}$.

(iii) {\bf Orthogonality.}

$\cu_{w\le 0}\perp \cu_{w\ge 0}[1]$.

(iv) {\bf Weight decompositions}.

 For any $M\in\obj \cu$ there
exists a distinguished triangle
\begin{equation}\label{wd}
B\to M\to A\stackrel{f}{\to} B[1]
\end{equation} 
such that $A\in \cu_{w\ge 0}[1],\  B\in \cu_{w\le 0}$.

II The category $\hw\subset \cu$ whose objects are
$\cu_{w=0}=\cu_{w\ge 0}\cap \cu_{w\le 0}$, $\hw(Z,T)=\cu(Z,T)$ for
$Z,T\in \cu_{w=0}$,
 will be called the {\it heart} of 
$w$.

III $\cu_{w\ge i}$ (resp. $\cu_{w\le i}$, resp.
$\cu_{w= i}$) will denote $\cu_{w\ge
0}[i]$ (resp. $\cu_{w\le 0}[i]$, resp. $\cu_{w= 0}[i]$).

IV We will  say that $(\cu,w)$ is {\it  bounded}  if
$\cup_{i\in \z} \cu_{w\le i}=\obj \cu=\cup_{i\in \z} \cu_{w\ge i}$.

V Let $\cu$ and $\cu'$ 
be triangulated categories endowed with
weight structures $w$ and
 $w'$, respectively; let $F:\cu\to \cu'$ be an exact functor.

$F$ will be called {\it left weight-exact} 
(with respect to $w,w'$) if it maps
$\cu_{w\le 0}$ to $\cu'_{w'\le 0}$; it will be called {\it right weight-exact} if it
maps $\cu_{w\ge 0}$ to $\cu'_{w'\ge 0}$. $F$ is called {\it weight-exact}
if it is both left 
and right weight-exact.

\end{defi}

\begin{rema}\label{rstws}

1. A  simple (and yet useful) example of a weight structure comes from the stupid
filtration on the homotopy categories
$K(B)\supset K^b(B)$ of cohomological complexes for an arbitrary additive category $B$. 
In this case
$K(B)_{w\le 0}$ (resp. $K(B)_{w\ge 0}$) is the class of complexes that are
homotopy equivalent to complexes
 concentrated in degrees $\ge 0$ (resp. $\le 0$).  The heart of this weight structure (either for $K(B)$ or for $K^b(B)$)
is  the Karoubi-closure  of $B$
 in the corresponding category. 

2. A weight decomposition (of any $M\in \obj\cu$) is (almost) never canonical.
Yet for an $m\in \z$ we will often need an (arbitrary) choice of a weight decomposition of $M[-m]$ shifted by $[m]$. This way we obtain a distinguished triangle \begin{equation}\label{ewd} w_{\le m}M\to M\to w_{\ge m+1}M \end{equation} 
with some $ w_{\ge m+1}M\in \cu_{w\ge m+1}$, $ w_{\le m}M\in \cu_{w\le m}$ (see Remark 1.2.2 of \cite{bws}); we will use this notation below (though $w_{\ge m+1}M$ and $ w_{\le m}M$ are not canonically determined by $M$, unless we impose some additional restrictions on these objects).

3. {\bf Caution on signs of weights.} When the author defined weight structures (in \cite{bws}), he considered $(\cu^{w\le 0}, \cu^{w\ge 0})$  such that $\cu^{w\le 0}$ is stable with respect to $[1]$ (similarly to the usual convention for $t$-structures); in particular, this meant that for $\cu=K(B)$  and for the 'stupid' weight structure for it mentioned above  a complex $C$ whose only non-zero term is the fifth one (i.e. $C^5\neq 0$) was 'of weight $5$'. Whereas this ({\it cohomological}) convention seems to be quite natural, for weights of mixed Hodge complexes, mixed Hodge modules (see Proposition 2.6 of \cite{btrans}), and mixed complexes of sheaves (see Proposition 3.6.1 of \cite{brelmot} and Proposition \ref{phuwe}(I) below) 'classically' exactly the opposite convention was used 
 (so, in this convention our $C$ is of weight $-5$). For this reason, in the current paper we use the 'reverse' ({\it homological}) convention for the signs of weights, that is compatible with the  'classical weights' (this convention for the Chow weight structure for motives was  used in \cite{hebpo},  
 in \cite{btrans}, and in \cite{brelmot}); so the signs of weights used below will be opposite to those in \cite{bws} and in \cite{bger}. 

\end{rema}

Now we recall those properties of weight structures that
will be needed below. 

\begin{pr} \label{pbw}
Let $\cu$ be a triangulated category.  $w$ will be a weight structure for
 $\cu$ everywhere except assertion \ref{idual}. 

\begin{enumerate}

\item\label{idual}
$(C_1,C_2)$ ($C_1,C_2\subset \obj \cu$) define a weight structure for $\cu$ if and only if
$(C_2^{op}, C_1^{op})$ define a weight structure for $\cu^{op}$.

\item\label{iext} 
  $\cu_{w\le 0}$, $\cu_{w\ge 0}$, and $\cu_{w=0}$
are extension-stable.

\item\label{iort} Let $w$  be a weight structure for
 $\cu$. Then $\cu_{w\ge 0}=(\cu_{w\le -1})^\perp$ and
 $\cu_{w\le 0}={}^\perp \cu_{w\ge 1}$ (see Notation).

\item\label{iwfun} Let $\cu$ and $\du$ be triangulated categories endowed with weight structures $w$ and $v$, respectively; let $w$ be bounded. Then an exact functor $F:\cu\to \du$  is left (resp. right) weight-exact if and only if $F(\cu_{w=0})\subset \du_{v\le 0}$ (resp. $F(\cu_{w=0})\subset \du_{v\ge 0}$).

\item\label{iuni} 
Suppose that 
$v$ is another weight structure for $\cu$; let $\cu_{v\le
0}\subset \cu_{w\le 0}$ and $\cu_{v\ge 0}\subset \cu_{w\ge 0}$.
Then $v=w$ (i.e. the inclusions are equalities).

\item \label{ipost} If $w$ is bounded, then $\cu_{w\le 0}$ is the smallest extension-stable subclass of $\obj \cu$ containing  $\cup_{i\le 0}\cu_{w=i}$; $\cu_{w\ge 0}$ is the smallest extension-stable class of $\obj \cu$ containing  $\cup_{i\ge 0}\cu_{w=i}$.


\end{enumerate}
\end{pr}
\begin{proof}

Most 
of the assertions were proved in \cite{bws} (pay attention to Remark \ref{rstws}(3)!); see 
 Proposition 1.2.3 of \cite{brelmot} for more details.
\end{proof}

\subsection{$t$-structures: a very short reminder and notation}
\label{dtst} 

To fix the notation we recall the definition of a $t$-structure.

\begin{defi}\label{dtstr}

A pair of subclasses  $\cu^{t\ge 0},\cu^{t\le 0}\subset\obj \cu$
 will be said to define a
$t$-structure $t$ if 
they satisfy the
following conditions:

(i) $\cu^{t\ge 0},\cu^{t\le 0}$ are strict i.e. contain all
objects of $\cu$ isomorphic to their elements.

(ii) $\cu^{t\ge 0}\subset \cu^{t\ge 0}[1]$, $\cu^{t\le
0}[1]\subset \cu^{t\le 0}$.

(iii) {\bf Orthogonality}. $\cu^{t\le 0}[1]\perp
\cu^{t\ge 0}$.

(iv) {\bf $t$-decompositions}.

For any $M\in\obj \cu$ there exists
a distinguished triangle
\begin{equation}\label{tdec}
A\to M\to B{\to} A[1]
\end{equation} such that $A\in \cu^{t\le 0}, B\in \cu^{t\ge 0}[-1]$.

\end{defi}

Bounded 
$t$-structures
can be defined similarly to Definition \ref{dwstr}(IV). 

We will need some more notation and properties
for $t$-structures.

\begin{defi} \label{dt2}

1. The category $\hrt$ whose objects are $\cu^{t=0}=\cu^{t\ge 0}\cap
\cu^{t\le 0}$, $\hrt(X,Y)=\cu(X,Y)$ for $X,Y\in \cu^{t=0}$,
 will be called the {\it heart} of
$t$. Recall 
that $\hrt$ is always
abelian; short exact sequences in $\hrt$ come from distinguished
triangles in
$\cu$.

2. $\cu^{t\ge l}$ (resp. $\cu^{t\le l}$) will denote $\cu^{t\ge
0}[-l]$ (resp. $\cu^{t\le 0}[-l]$).

\end{defi}

\begin{rema}\label{rts}
1. Recall 
that (\ref{tdec})
defines additive functors $\cu\to \cu^{t\le 0}:M\mapsto A$ and $C\to
\cu^{t\ge 1}:M\mapsto B$. We will denote $A,B$ by $M^{\tau\le 0}$ and
$M^{\tau\ge 1}[-1]$, respectively. (\ref{tdec}) will be called the
{\it $t$-decomposition} of $M$.

More generally, the $t$-components of $M[i]$ (for any $i\in \z$) will be denoted by
$M^{\tau\le i}\in
\cu^{t\le 0}$ and $M^{\tau\ge i+1}[-1]\in \cu^{t\ge 1}$, respectively.

$\tau_{\le i} M$ will denote $M^{\tau\le i}[-i]$; $\tau_{\ge i} M$ will denote $M^{\tau\ge i}[-i]$.

2. 
The functor $M\mapsto \tau_{\ge 0} M$ is  left adjoint to the inclusion $\cu^{t\ge 0}\to \cu$.

3. We will also need the following easy (and well-known) properties of $t$-structures.

 The first one is  Proposition 1.3.17(iii) of \cite{bbd}: if a 
 functor $F$ is left adjoint to $G$, and their targets are endowed with $t$-structures, then $F$ is right $t$-exact if and only if $G$ is left $t$-exact.  The latter assertions mean that $F$   respects '$t$-negative' objects, whereas $G$ respects $t$-positive ones. 

The second property is: if for two $t$-structures $t$ and $t'$ on a triangulated $\cu$ the identity functor is $t$-exact (for the pairs $(\cu,t)$ and $(\cu,t')$, i.e. $\cu^{t\le 0}\subset \cu^{t'\le 0}$ and $\cu^{t\ge 0}\subset \cu^{t'\ge 0}$), then $t=t'$. Indeed, the previous statement yields that the identity is also $t$-exact as a functor from $(\cu,t')$ to $(\cu,t)$.

4. 
We denote by $H^t_{0}$ the zeroth homology functor corresponding to
$t$. 
Shifting the $t$-decomposition of $M^{\tau\le 0}[-1]$ by
$[1]$ we obtain  a canonical and functorial (with respect to $M$)
distinguished triangle $\tau_{\le -1} M \to \tau_{\le 0} M\to H^t_{0}(M)$; we denote $H^t_{0}(M[i])$ by $H^t_{i}(M)$.

If a $t$-structure is non-degenerate (i.e. if $\cup_{i\in\z}\cu^{t\le i}=\cup_{i\in\z}\cu^{t\ge i}=\ns$; note that this is certainly the case for bounded $t$-structures) then the collection of $H^t_{i}$ is conservative. Moreover, in this case for $C\in \obj \cu$ we have: $C\in \cu^{t\le 0}$ (resp. $C\in \cu^{t\ge 0}$) whenever $H^i(C)=0$ for all $i>0$ (resp. for all $i<0$).

\end{rema}

\subsection{On weight filtrations and (degenerating) weight spectral sequences}\label{swss}

Now we recall certain properties of weight filtrations and weight spectral sequences. Most of them were established in \S2 of \cite{bws}, whereas the degeneration of weight spectral sequences was studied in \S3 of \cite{btrans}.

Let $\au$ be an abelian category. In \S2 
of \cite{bws}
for $H:\cu\to \au$ that is either cohomological or homological (i.e. it is either covariant or contravariant, and converts distinguished triangles into long exact sequences) certain {\it weight filtrations} and {\it weight spectral sequences}  (corresponding to $w$) were introduced. Below we will be more interested in
the homological functor case; certainly, one can pass to cohomology 
by a simple reversion of arrows (cf. \S2.4 of ibid.).

\begin{defi}\label{dwfil}
Let $H:\cu\to \au$ be a 
covariant functor, $i\in \z$.

1. We denote $H\circ [i]:\cu\to \au$ by $H_i$.

2. We choose some $w_{\le i}M$ and define  the {\it weight filtration} for $H$  by $W_iH:M\mapsto \imm (H(w_{\le i}M)\to H(M))$.

Recall that $W_iH$ is functorial in $M$ (in particular, it does not depend on the choice of  $w_{\le i}M$); see Proposition 2.1.2(1) of ibid.


\end{defi}

Now we recall some of the properties of weight spectral sequences; we are especially interested in the case when they degenerate. 

\begin{pr}\label{pwss}
 
I For a homological $H:\cu\to \au$ and any $M\in \obj \cu$  there exists a spectral sequence $T=T_w(H,M)$ with $E_1^{pq}(T)=
H_q(M^{p})$ for certain $M^m\in \cu_{w=0}$ (coming from certain weight decompositions as in (\ref{ewd})) that converges to 
$E_{\infty}^{p+q}=H_{p+q}(M)$.
 $T$ is  
$\cu$-functorial  in $M$ and in $H$ (with respect to  composition of $H$ with exact functors of abelian categories) 
starting from $E_2$. 
Besides, the step of filtration given by ($E_{\infty}^{l,m-l}:$ $l\ge k$)
 on $H_{m}(M)$ (for some $k,m\in \z$) equals  $(W_{-k}H_{m})(M)$.
Moreover, $T(H,M)$ comes from an exact couple with $D_1^{pq}=H_{p+q}(w_{\le - p}M)$ (here one can fix any choice of $w_{\le -p}M$). 

We will say that $T$ {\it degenerates at $E_2$} (for a fixed $H$) if $T_w(H,M)$ does so for any $M\in \obj \cu$.

II Suppose that $T$ degenerates at $E_2$ (as above), $i\in \z$. Then the following statements are fulfilled.

1. The functors $W_iH$ and $W_i'H:M\mapsto H(M)/W_{i-1}H(M)$ are homological.

2. For any $f\in \cu(M,Y)$ the morphism $H(f)$ is strictly compatible with the filtration of $H$ by $W_i$ i.e. $W_{i}H(M)$ surjects onto $W_{i}H(Y)\cap \imm H(f)$.   



\end{pr}

\begin{proof}

Immediate from Proposition 3.1.2 of \cite{btrans} (recall that $w$ is bounded by our convention).
 
\end{proof}

\begin{rema}\label{rwss}

The description of the exact couple for $T_w(H,M)$ (that can be found in loc. cit.) also easily yields the following results.

1. If $H'=H\circ [d]$ then $T_w(H',M)$ for any $M\in \obj \cu$ can be obtained by a 'shift' from $T_w(H,M)$ i.e. one can take $E_r^{p,q}T_w(H',M)=E_r^{p,q+d}T_w(H, M)$ for all $r\ge 1,\ p,q\in\z$, and the differentials behave in the similar way.  Certainly, this yields a functorial description of $T_w(H',M)$ starting from $E_2$.

2. The functoriality of $T_w(H,M)$ described in assertion I of the proposition can be easily generalized as follows.
Let \begin{equation}\label{edeg} \begin{CD}
\cu@>{F}>>\cu'\\
@VV{H}V@VV{H'}V \\
\au@>{G}>>\au'
\end{CD}\end{equation}
be a commutative square of functors, where $\cu'$ is a triangulated category endowed with a weight structure $w'$ such that $F$ is a weight-exact (and exact) functor, $G$ is an exact functor of abelian categories. Then for any $M\in \obj \cu$ one can obtain $T_{w'}(H',F(M))$  by applying $G$ (termwise) to  $T_{w}(H, M)$. Again, this is a functorial isomorphism starting from $E_2$.

3. Hence in this setting for any $d\in \z$ we obtain: if the spectral sequence $T_{w}(H,M)$ degenerates at $E_2$, then $T_{w'}(H'\circ [d], F(M))$ also does. The converse statement is also true if $G$ is conservative.

This is a certain generalization of Proposition \ref{pdeg}(II) below.

\end{rema}


Now we introduce the notion of a weight filtration for an abelian category following Definition D.7.2 of \cite{bvk}.  

\begin{defi}\label{dwfilt}

For an abelian $\au$, we will say that an increasing family of full subcategories $\au_{\le i}\subset \au$, $i\in \z$, yields a {\it weight filtration} for $\au$ if $\cap_{i\in \z}\au_{\le i}=\ns$, $\cup_{i\in \z}\au_{\le i}=\au$, and there exist exact right 
adjoints $W_{\le i}$ to the embeddings  $\au_{\le i}\to \au$.
\end{defi}

We will need the following statement.

\begin{lem}\label{lwfil}
 
 Let  $\au_{\le m}$, $m\in \z$, yield a weight filtration for $\au$. 
 Then the following statements are valid.
 
1. $\au_{\le m}$ are  exact abelian subcategories of $\au$. 

2. All $W_{\le m}$ are idempotent endofunctors.

3. The adjunctions yield functorial embeddings of $W_{\le m}M\to M$ such that $W_{\le m-1}M\subset W_{\le m}M$ for all $m\in \z$, and
the functors $W_{\ge m}:M\mapsto M/W_{\le m-1}M$ are exact also. 

4. The
 categories $\aum$ being the 'kernels' of the restriction of $W_{\le m-1}$ to $\au_{\le m}$, are abelian, and $\aum\perp \au_j$ for any $j\neq m$. 
   
\end{lem}
\begin{proof}
This is (a part of) Lemma 2.1.2  of \cite{btrans}.


\end{proof}

Now we fix certain (bounded) $w$ and $t$ for $\cu$, and study a  condition  ensuring that $w$ induces a weight filtration for $\hrt$.


\begin{pr}\label{pdeg}
Let $H=H^t_0$. 

I Suppose that the corresponding $T$ degenerates. Then 
the functors $W_iH:\cu\to \hrt$ are homological. 
 The restrictions $W_{\le i}$ of $W_iH$ to $\hrt$ define a weight filtration for this category. Besides, $W_iH 
 =W_{\le i}\circ H$. 

II Let $\bu$ be an abelian category; let $F:\hrt\to \bu$ be an exact functor.

1. Suppose that  $T$ degenerates. Then   $T_w(F\circ H,-)$ also does.

2. Conversely,  suppose that $F$ is conservative and that $T_w(F\circ H,-)$ degenerates.  Then $T$ degenerates. 

Moreover, for $M\in \cu^{t=0}$ we have: $W_{\le i}M=M$ (resp. $W_{\le i}M=0$) if and only if $W_i(F\circ H)(M)=F(M)$ (resp. $W_i(F\circ H)(M)=0$).

\end{pr}
\begin{proof}

 This is 
 (a part of) Proposition 3.2.1 
 of \cite{btrans}.%

\end{proof}

\subsection{On transversal weight and $t$-structures}\label{strans}

Let $t$ be a $t$-structure for $\cu$, and $w$ be a weight structure for it. 

\begin{defi}\label{dnwd}
1. For some $\cu,t,w$ we will say that a distinguished triangle (\ref{ewd}) (for some $m,M$) is {\it nice} if 
$w_{\le m}M,M,w_{\ge m+1}M\in \cu^{t=0}$.

We will also say that this distinguished triangle is a {\it nice decomposition} of $M$ (for the corresponding $m$).

2. Suppose (as we always do) that $t$ and $w$ are bounded.

We will say that $t$ and $w$ are transversal if 
a nice decomposition exists for any $m\in \z$ and any $M\in \cu^{t=0}$.


\end{defi}

\begin{pr}\label{ptrans}
I We fix some $\cu,w,t,m$; suppose that for a certain $N\subset \cu^{t=0}$ a nice decomposition  exists for any $M\in N$. Consider $N'\subset \cu^{t=0}$ being the smallest subclass containing $N$ that 
satisfies the following condition: if $A,C\in N'$, 
\begin{equation}\label{ecom}
A\stackrel{f}{\to}B\stackrel{g}{\to}C\end{equation}
is a complex (i.e. $g\circ f=0$), $f$ is monomorphic, $g$ is epimorphic, $ \ke g/\imm f\in N'$, then $B\in N'$. Then a nice choice of (\ref{ewd}) exists for any $M\in N'$.

II If $t$ is transversal to $w$, then the following statements are fulfilled for any $i\in\z$, $M\in \cu^{t=0}$, $Y\in \obj\cu$. 

1. For any $H$ that could be presented as $F\circ H^t_m$, where $F:\hrt\to \au$ is an exact 
functor, 
$T_w(H,-)$ degenerates at $E_2$.

2. Nice decompositions  exist and are $\hrt$-functorial in $M$ (for a fixed $m$). The corresponding  functor  $W_{\le m}:M\mapsto w_{\le m} M$  can be described as (the restriction to $\hrt$ of) $W_{m}H^t_0$ (see Definition \ref{dwfil}(2));
i.e. it coincides with the functor $W_{\le m}$ given by Proposition \ref{pdeg}(I).

3. The category $\au_{m}=\cu^{t=0}\cap \cu_{w= m}$ is (abelian) semisimple; 
there is a splitting $\cu_{w=0}=\bigoplus_{m\in\z} \obj \au_m[-m]$ given by $Y\mapsto \bigoplus_{i\in \z}H_i^t(Y)[-i]$.

4. $W_{\le m}M$ yield an increasing filtration for $M$ whose $m$-th factor belongs to $\au_m$. Moreover, this filtration is uniquely and functorially determined by this condition. 


5. $Y\in \cu_{w\le m}$ (resp. $Y\in \cu_{w\ge m}$) if and only if for
 any $j\in \z$ we have $W_{\le m+j} (H^t_j(Y))=H^t_j(Y)$  (resp. $W_{\le m+j-1} (H^t_j(Y))=0$; note that this is equivalent to $W_{\ge m+j} (H^t_j(Y))=H^t_j(Y)$).

III For any $t,w$, $m\in\z$, 
and any choice of  
$w_{\le m}M$ (and a morphism $w_{\le m}M\to M$ corresponding to a weight decomposition) consider the morphism $f_m(M):(W_m H_0^t)(M)\to M$ (cf. 
Definition \ref{dwfil}(2)).
Then $t$ is transversal to $w$ if and only if this morphism extends to a nice decomposition (for any $M,m$).

IV For a family of semisimple (abelian)  $\{\au_m \subset \cu,\ m\in\z \}$, 
suppose that $\lan\cup_{m\in \z} \obj \au_{m}\ra=\cu$, and $\au_m\perp \au_j[s]$ for any $m,j,s\in \z$ such that: either $s<0$, or $s>m-j$, or
 $s=0$ and   $m> j$.  

Then there exist transversal $w$ and $t$ such that $\cu_{w=0}=\bigoplus_m\obj \au_m[-m]$, and $\hrt$ is the smallest extension-stable subcategory of $\cu$ containing $\cup \au_m$.

V For any $t,w$, $t$ is transversal to $w$ if and only if there exists a family of semisimple  abelian $\aum\subset \hrt$ ($m\in \z$) such that: 
$\obj \aum\cap \obj \au_j=\ns$ for all $j\neq m$, $j\in \z$, 
 and $\cu_{w=0}=\bigoplus_m\obj \aum[-m]$. 

\end{pr}
\begin{proof}

I This is Lemma 1.1.3 of \cite{btrans}.

II 1. Immediate from Proposition 3.2.1(II,III.1) of ibid.

2. The functoriality of nice decompositions is the condition (iv') of Theorem 1.2.1 of ibid. (which is equivalent to condition (iv) of loc. cit. that we took above for the definition of transversality). The 
equality of two distinct descriptions of $W_{\le m}$ is given by Proposition 3.2.1(II) of ibid.

3. See Remark 1.2.3(2) of \cite{btrans}.

4. Immediate from condition (iii') of Theorem 1.2.1 of ibid. (cf. the proof of assertion II2). 

5. This is Proposition  1.2.4(I2) of ibid. 

III See Remark 1.2.3(4) of ibid. 

IV Theorem 1.2.1 of ibid. implies that such a family of $\aum$ does yield some transversal structures $t,w$. Besides, loc. cit. also allows to calculate $\hrt$, whereas  $\cu_{w=0}=\bigoplus_m\obj \aum[-m]$ by Remark
1.8(2) of ibid.

V Since $\lan\hw\ra=\cu$,
we obtain that $\lan\cup_{m\in \z} \obj \aum\ra=\cu$. 
Since $\aum$ are semisimple, we obtain that $\au_m\perp \au_j$ for any $m\neq j$.
The orthogonality axioms of weight and $t$-structures also yield the remaining orthogonality conditions that are needed in order to apply the previous assertion.   We obtain that certain transversal $t'$ and $w'$ exist; besides, $\hw'=\hw$ and $\hrt'\subset \hrt$. Since
 $w'$ is bounded, this implies $w=w'$ (see Proposition \ref{pbw}(\ref{iwfun},\ref{iuni})).  Since $t'$ is bounded also, we easily deduce that $t'=t$.


\end{proof}

\begin{rema}\label{rchk}
In the case of motives (of smooth projective varieties over a field) the splittings mentioned in assertions II3 and IV corresponds to the so-called {\it Chow-Kunneth} decompositions ('of the diagonal').
\end{rema}

\subsection{Some auxiliary 'gluing statements'}
\label{sglu}

Below we will apply several gluing arguments. We chose to gather the  definitions and auxiliary statements related with this matter here.

\begin{defi}\label{dglu}
1. The $9$-tuple $(\cu,\du,\eu, i_*, j^*, i^*, i^!, j_!, j_*)$  is called a
{\it gluing data} if it satisfies the following conditions.

(i) $\cu,\du,\eu$ are triangulated categories; $i_*:\du\to\cu$,
$j^*:\cu\to\eu$, $i^*:\cu\to\du$, $i^!:\cu\to\du$, $j_*:\eu\to\cu$,
$j_!:\eu\to\cu$ are exact functors.

(ii) $i^*$ (resp. $i^!$) is left (resp. right) adjoint to $i_*$;
$j_!$ (resp. $j_*$) is left (resp. right) adjoint to $j^*$.

(iii)  $i_*$ is a full embedding; $j^*$ is isomorphic to the
localization (functor) of $\cu$ by
$i_*(\du)$.

(iv) For any $M\in \obj \cu$ the pairs of morphisms $j_!j^*M \to M
\to i_*i^*M$ and $i_*i^!M \to M \to j_*j^*M$ can be completed to
distinguished triangles (here the connecting
morphisms come from the adjunctions of (ii)).


(v) $i^*j_!=0$; $i^!j_*=0$.

(vi) All of the adjunction transformations $i^*i_*\to \id_{\du}\to
      i^!i_*$ and $j^*j_*\to \id_{\eu}\to j^*j_!$ are isomorphisms of
      functors.

2. In the setting of part 1 of this definition, we will say that $M\in \obj \cu$ is a {\it lift} of an $Y\in \obj \eu$ if $j^*M\cong Y$. Similarly, for a lift of a distinguished triangle $C$ in $\eu$ is a distinguished triangle $C'$ in $\cu$ such that $j^*C'\cong C$.

3. In the setting of part 1, suppose that $\cu$ is endowed with a $t$-structure $t=t_{\cu}$. We define the {\it intermediate image} functor $j_{!*}:\eu\to \cu$ as $M\mapsto \imm(H_0^{t_{\cu}}j_!M\to H_0^{t_{\cu}}j_*M)$; here the morphism $j_!\to j_*$ comes from adjunction (and we use the fact that $j^*j_!\cong j^*j_*\cong 1_{\eu}$; cf. (1.4.6.2) and Definition 1.4.22 of \cite{bbd}). 

4. In the setting of part 1, suppose also that $\du$ and $\eu$ are endowed with  certain $t$-structures $t_{\du}$ and $t_{\eu}$, respectively. Then we will say that a $t$-structure $t=t_{\cu}$ for $\cu$ is {\it glued from}  $t_{\du}$ and $t_{\eu}$ if 
we have: $ \cu^{t_{\cu}\le 0}=\{M\in \obj \cu:  j^*M\in \eu^{t_{\eu}\le 0}\text{ and }  
i^*M\in \du^{t_{\du}\le 0}\}$, and $ \cu^{t_{\cu}\ge 0}=\{M\in \obj \cu:  j^*M\in \eu^{t_{\eu}\ge 0}\text{ and }  
i^!M\in \du^{t_{\du}\ge 0}\}$.  In this case we will also say that $\cu$, $\du$, and $\eu$ are {\it endowed with compatible $t$-structures}.


\end{defi}

\begin{rema}\label{rglu}
 Our definition of a gluing data is far from being the 'minimal' one. Actually, it is 
 well known (see Chapter 9 of \cite{neebook}) that
a gluing data can be uniquely recovered from an inclusion
$\du\to \cu$ of triangulated categories that admits both a left
and a right adjoint functor.

Our notation for the connecting functors 
is (certainly) coherent with Proposition \ref{pcisdeg}(\ref{iglu}) below.

\end{rema}


\begin{pr}\label{pglu}
I In the setting of Definition \ref{dglu}(1) assume that $\du$ is endowed with a $t$-structure $t_{\du}$. Then for any $M\in \obj \cu$ any distinguished triangle $A'\to M'(= j^*M)\to B'$  (in $\eu$) possesses a lift $A\to M\to B$ (see Definition \ref{dglu}(2))
such that $i^*A\in \du^{t_{\du}\le 0}$, and  $i^!B\in \du^{t_{\du}\ge 1}$.

II In the setting of Definition \ref{dglu}(4) the following statements are fulfilled.

1. There exists a $t$-structure  $t_{\cu}$ for $\cu$ glued from  $t_{\du}$ and $t_{\eu}$.

2. $t_{\cu}$ is characterized by the following property:  $i_*$ and $j^*$ are $t$-exact. 

Moreover,  
$j_!$ and $i^*$ are right $t$-exact (see Remark \ref{rts}(3)), whereas $j_*$ and $i^!$ are left $t$-exact (with respect to $t_{\du}$, $t_{\cu}$, and $t_{\eu}$, respectively).


III In the setting of Definition \ref{dglu}(1) assume that
$\cu,\du,\eu$ are  endowed with weight structures $w_{\cu},\ w_{\du}$, and $w_{\eu}$, respectively, and that $i_*$ and $j^*$ are weight-exact. Then we will say that $w_{\cu},\ w_{\du}$, and $w_{\eu}$ are {\it compatible}. 

In this situation
$j_!$ and $i^*$ are left weight-exact, whereas $j_*$ and $i^!$ are right weight-exact.
Besides, we have: $\cu_{w\ge 0}
=\{M\in \obj
\cu:\ i^!M\in \du_{w_\du\ge 0} ,\ j^*M\in \eu_{w_\eu\ge 0} \}$ and
$\cu_{w\le 0}
=\{M\in \obj \cu:\ i^*M\in \du_{w_\du\le 0} ,\ j^*M\in
\eu_{w_\eu\le 0} \}$.

IV Assume that $\cu$, $\du$, and $\eu$ are endowed with compatible $t$-structures (see Definition \ref{dglu}(4)). Then for any $M,Y\in \cu^{t=0}$, $M',Y'\in \eu^{t=0}$, $i\in \z$ the following statements are valid.

\begin{enumerate}

\item\label{ilift} $j^*j_!M'\cong j^*j_{!*}M'\cong j^*j_{*}M'\cong M'$. 

\item\label{iim} $i^*j_{!*}M'\cong \tau_{\le -1}i^*j_{*}M'$, and $i^!j_{!*}M'\cong \tau_{\ge 1} i^!j_{!}M'$. 

\item\label{icohom} If a complex $A\to B\to C$ (in $\hrt_{\du}$) is exact in the term $B$, then the middle-term homology object of the complex  $j_{!*}A\to j_{!*}B\to j_{!*}C$ belongs to $i_*\du^{t=0}$. 


\item\label{ibiext} $M$ can be obtained from $j_{!*}j^*M$ via two extensions by elements of $i_*\du^{t=0}$. 

\item\label{imonoepi} $j_{!*}$ maps monomorphisms to monomorphisms, and epimorphisms to epimorphisms.

\item\label{iiim} $j_{!*}M'$ does not have non-trivial subobjects of factor-objects belonging to $i_*\hrt_{\du}$.

\item\label{issm} 
The homomorphism $\cu(j_{!*}M',j_{!*}Y')\to \eu(M',Y')$ induced by $j^*$ is bijective.

\item\label{isso} If $M'$ is simple, $j_{!*}M'$ also is.

\item\label{isss} If $M'$ is semisimple, then $j_{!*}M'$ can be functorially characterized as a semisimple lift of $M'$ none of whose components are  killed by $j^*$.

\end{enumerate}

\end{pr}
\begin{proof}
I We argue as in the proof of Theorem 1.4.10 of \cite{bbd}. We consider $Y=\co (M\to j_*B')[-1]$ and  
$A=\co(Y\to i_*(\tau_{\du, \ge 1}i^*Y))[-1]$. 
We complete the commutative triangle $A\to Y\to M$ to  an octahedral diagram 
$$
\xymatrix{ j_*B' \ar[rd]^{[1]}\ar[dd]^{[1]} & & M \ar[ll] \\ & Y \ar[ru]\ar[ld] & \\ i_*\tau_{\du,\ge 1}i^*Y \ar[rr]^{[1]} & & A\ar[lu]\ar[uu] }  
$$
and denote its sixth vertex by $B$.

Now we argue exactly as in loc. cit. (using the fact that exact functors convert distinguished triangles into distinguished ones, and the 'axioms' of a gluing data). %
We obtain that $j^* (i_*\tau_{\du,\ge 1}i^*Y\to  B\to j_*B')\cong (0\to j^*B\to B')$; hence $j^*B\cong B'$. 
Next, $j^*(A\to M\to B)\cong (j^*A\to M'\to B')$; hence $j^*A\cong A'$.
Furthermore, $i^*(A\to Y\to i_*\tau_{\du,\ge 1}i^*Y)\cong (i^*A\to i^*Y\to \tau_{\du,\ge 1}i^*Y)$; hence $i^*A\cong \tau_{\du,\le 0}i^*Y$. It remains to note that $i^! (i_*\tau_{\du,\ge 1}i^*Y\to  B\to j_*B')\cong (\tau_{\du,\ge 1}i^*Y\to  i^*B\to 0)$; hence $i^!B\cong \tau_{\du,\ge 1}i^*Y$.

 II.1. This is (exactly)  Theorem 1.4.10 of \cite{bbd}. 
 
 II2. Obviously, if $t$ is glued from $t_{\du}$ and $t_{\eu}$, then $j^*$ is $t$-exact. Since $j^*i_*=0$, the adjunctions to $i_*$ also yield that $i_*$ is $t$-exact; see 
 Remark \ref{rts}(3).
 
 Now, suppose that the $t$-exactness of $i_*$ and $j^*$ is  fulfilled for some 
 $t$-structure $t'$ for $\cu$. 
 Then loc. cit. 
yields all of our $t$-exactness statements  for $t'$ (and so, they are fulfilled for  $t$). It follows that  $1_{\cu}$ is $t$-exact as a functor from $(\cu,t')$ to $(\cu,t)$. Applying the other statement in loc. cit., 
we obtain that $t=t'$.  

III Immediate from Proposition 1.2.3(13,15) of \cite{brelmot}.

IV The proofs are 
easy applications of the results of (the end of) \S1.4 of \cite{bbd}.

(IV\ref{ilift})  
is immediate from the  axioms of a gluing data and the $t$-exactness  of $j^*$.

(IV\ref{iim}) is immediate from Proposition 1.4.23 of ibid. 

(IV\ref{icohom}): The previous assertion yields that the middle term homology in question is killed by $j^*$. Since the categorical kernel of $j^*$ is $i_*\du$, and $i_*$ is $t$-exact, we obtain the result.

(IV\ref{ibiext}): By assertion IV\ref{ilift}, we have a $\hrt_{\cu}$-epimorphism $a:H_{0}^{t_{\cu}}j_{!}M\to j_{!*}M$, and a $\hrt_{\cu}$-monomorphism $b:j_{!*}M\to H_{0}^{t_{\cu}}j_{*}M$; both of them become isomorphisms after the application of $j^*$. Besides, adjunctions yield that $b\circ a$ factors through $M$. As in the proof of (IV\ref{icohom}), the result follows immediately.

(IV\ref{imonoepi},\ref{iiim}): Immediate from Corollary 1.4.25 of ibid.

(IV\ref{issm}) is an easy consequence of (IV\ref{iiim}). Indeed, since $j^*j_{!*}\cong 1_{\hrt_{\eu}}$, it suffices to verify that the homomorphism $\cu(j_{!*}M',j_{!*}Y')\to \eu(M',Y')$ induced by $j^*$ is injective. Let $f $ be a non-zero element of $ \cu(j_{!*}M',j_{!*}Y')$. Then assertion (IV\ref{iiim}) yields that $\imm f$ is not isomorphic to an object of $i_*(\hrt_{\du})$. Hence $j^*\imm f\neq 0$; since $j^*$ is $t$-exact we obtain that $j^*f\neq 0$.

(IV\ref{isso}): This is just Proposition 1.4.26 of ibid. 

(IV\ref{isss}): 
We may assume that $M'$ is simple. Then $j_{!*}M$ is simple also by the previous assertion. Assertion IV\ref{ibiext} yields that $j_{!*}M$ is the only simple lift of $M'$. Lastly, assertion IV\ref{iiim} implies that this characterization of $j_{!*}M$ is functorial. 

\end{proof}

\begin{rema}
So, $j_{!*}M'$ is the 'minimal' lift of $M'$. 
 As a consequence, when  we will 'lift nice decompositions' (in the proof of Theorem 
\ref{tvmts} below)  it will be sufficient to check whether $j_{!*}$ 'respects weights'. In order to verify the latter assertion, we will apply Theorem \ref{tdeg}(II).

\end{rema}

\section{On relative motives and $\ql$-sheaves}\label{srrelmot} 

In \S\ref{scheme} we introduce certain 
 terminology 
for schemes and their morphisms; we also discuss our restrictions on base schemes.

In \S\ref{sbrmot}  we recall some of basic properties of Beilinson motives over $S$ (as defined in \cite{degcis}). 

In \S\ref{schow} we recall certain properties of the Chow weight structure $\wchow$ for $\dmcs$ (as introduced in \cite{hebpo} and \cite{brelmot}); we also prove some new 'continuity' properties of this weight structure.

In \S\ref{setre} we treat the \'etale realization of $S$-motives and the perverse $t$-structure for its target.

In \S\ref{swsshetl} we study weights for mixed sheaves and relate them with (the degeneration of) Chow-weight spectral sequences for $\hetlz$. The latter degenerate at $E_2$  if $S$ is a very reasonable scheme (we conjecture that they degenerate for a general reasonable $S$ also). This yields that the   Chow-weight filtration for such (co)homology  is strictly restricted by ('motivic') morphisms.

\subsection{Schemes and morphisms: some terminology and a discussion of restrictions}\label{scheme}

All morphisms and  schemes below will be separated. Besides, all schemes will be excellent Noetherian of finite Krull dimension. 
$S$ will usually be our 
base scheme. Often $j:U\to S$ will be an open immersion, and $i:Z\to S$ will be the complementary closed embedding.


Below $l$ will always be a prime number (as well as $l'$); we will usually assume $l$ to be fixed. $p$ will usually denote the characteristic of some scheme (so it is either a prime number or $0$); usually $p\neq l$.
We will say that $p$ is the characteristic of $S$ (only) if it is an equicharacteristic $p$ scheme (so it is an $\sfp$-scheme if $p>0$ and a $\sq$-one for $p=0$).

Below we will identify a Zariski point (of a scheme $S$) with the spectrum of its residue field (sometimes we will also make no distinction between the spectrum of a field and the field itself). $\sss$ will denote the set of 
(Zariski) points of $S$. For 
$K\in \sss$ we will denote the natural morphism $K\to S$ by $j_K$. 
We will call the dimension of the closure of $K$ in $S$  the dimension of $K$.

Now we introduce some terminology for schemes and their morphisms.

\begin{defi}\label{dmor}   
1. We will call a 
scheme $S$ {\it reasonable} if it is
of finite type over some  (Noetherian excellent separated)  
  regular scheme $S_0$ of dimension lesser than or equal to $1$.
  
  We will only consider reasonable schemes below. For most of them one can assume $S_0$ to be fixed (yet we will often consider Zariski points of our schemes).
  In particular, when we will say that a morphism of schemes is of finite type we will always 
  assume that we have chosen a common $S_0$ for them.
  
 2. A morphisms $g:X'\to Y$ will be called {\it essentially pro-affine} if it factorizes as $X'\stackrel{h}{\to}X\stackrel{f}{\to} Y$ where $f$ is a finite type morphism, $h$ is the  inverse limit of a filtered system of affine 
 morphisms $h_i:X_i\to X$.
 
 3. An essentially pro-affine morphism $g':X\to Y$ will be called  
 {\it quasi-regular} if all the corresponding $f_i:X_i\to Y$ are compositions of chains of finite type smooth morphisms and finite universal homeomorphisms.

 4. A reasonable scheme $S$ will be called {\it very reasonable} if there exists a surjective finite type smooth morphism  $f:S'\to S$ such that $S'$ 
 can be presented as the inverse limit of some filtered system of schemes $S'_i$ that are of finite type over (the spectrum of) some field $K$ and such that all the transition morphisms $S'_i\to S'_j$ are smooth affine.

 \end{defi}

\begin{rema}\label{reas}
I.1. Obviously, 
any morphism of spectra of fields is quasi-regular, whereas any (separated) finite type scheme over the spectrum of a field is very reasonable.

2. Certainly, 
essentially affine and quasi-regular morphisms are stable with respect to base change. 

3. Being very reasonable is an (\'etale) local property; this is why we did not assume that $S=S'$ in the definition above. We will use this fact in the proof of Theorem \ref{tdeg}(II) below in order to reduce general very reasonable schemes to irreducible ones.

II By the celebrated theorem of Popescu (see Theorem 1.8 of \cite{pop} or Theorem 4.1.5 of \cite{degcis}) all {\it regular} morphisms of Noetherian affine schemes 
are quasi-regular (this result motivated our choice of the term).

III We have three reasons to restrict ourselves to reasonable schemes (in this paper). Yet possibly  our results can  be extended  to arbitrary excellent Noetherian separated schemes of finite Krull dimension. Now we explain this in more detail.

1. The 
Chow motives over $S$ are only known to yield the heart of a weight structure if $S$ is of finite type over an $S_0$ of dimension $\le 2$. Yet this restriction can be avoided: in \S2.3 of \cite{brelmot} an 'alternative' construction of the Chow weight structure was described. In loc. cit. it was proved that this version of $\wchow$ possesses all the properties listed in Theorem \ref{twchow}(II-IV) below, whereas (the new) parts V-VI of the theorem can be established using the explicit 'generators' of $(\dmcs_{\wchow\le 0},\dmcs_{\wchow\le 0})$ given by Proposition 2.3.4(I2) of ibid.

2. The existence of such an $S_0$ is also required (in Theorem 5.8.12(2) of \cite{cdet}) in order to ensure the existence of a dualizing  object in $\dmhsz$. Yet we only need (for the "h-version" of the \'etale realization of $S$-motives; see \S\ref{setre} below) the existence of a dualizing object in the category
$\dbcszl$ defined in \S5.9.19 of ibid. Possibly the latter fact can be deduced from Theorem XVII.0.9 of \cite{illgabb}.

3. Quite probably the results of ibid. 
also yield the existence of a 'reasonable' $\dshsl$ together with a (self-dual) perverse $t$-structure for it for a not necessarily reasonable $S$. Yet the author has never met any claims of this sort in the literature. On the other hand, in Theorem 6.3 of \cite{eke} says that  the arguments  of \cite{bbd} (for $l$-adic \'etale sheaves over finite type $\sfp$-schemes) carry over to reasonable schemes.


IV We define the class of very reasonable schemes since we can prove 
Theorem \ref{tdeg}(II) for them (and this is crucial for the main results of this paper). Now we discuss possible generalizations of this statement.


1. 
One can reduce Conjecture \ref{cdeg} for $S$ to that for its base changes to the completions of $S_0$ at closed points. Next, our method of the proof of Theorem \ref{tdeg}(II) would require 'approximating' the corresponding $S'$ by schemes $S'_i$ such that certain 'weights' are defined for (some version of) perverse sheaves over $S_i$.

Yet the author does not know of any weights of this sort in the case when $S'_i$ are not equicharacteristic schemes (even if they are the spectra of complete discrete valuations rings); this problem seems to be related with Deligne's weight-monodromy conjecture.

2. Moreover, the argument we use below does not work (even) for $S'$ being the spectrum of $K[[t]]$ ($K$ is a field). The problem here is the following one: though any motif $M$ over $S'$ has a 'model' $M_R$ over some $\spe R$, $R$ is finitely generated over $K$ (see Proposition \ref{pcisdeg}(\ref{icont})), the morphism $S'\to \spe R$ does not even have to be equidimensional (since the dimension of $R$ can be arbitrarily large). Possibly, one can apply the method used for the proof of Proposition 5.1 of \cite{ito}, and consider the pullback $M_{R'}$ of $M_R$ to a one-dimensional factor $R'$ of $R$. Next in order to establish the degeneration of $T_{\wchow(S')}(\hetl, M)$ one should apply the proper base change theorem (somehow, in order to relate $M$ with $M_{R'}$) and (probably) consider a perversity for $\spe R$ that is not self-dual (so that the  higher perverse  
inverse images with respect to morphisms $S'\to \spe R$ and $\spe R'\to \spe R$ of the terms of the Chow-weight spectral sequence for perverse \'etale homology of $M_R$ would vanish). One may say that such an $M_{R'}$ is a "clever model" for $M$ (a sort of Artin approximation for $R$).

Possibly, the author will study these questions (and relate them with Rappoport-Zink spectral sequences) in a subsequent paper.

\end{rema}

All the motives that
we will consider in this paper will have  rational coefficients (so that we will not mention rational coefficients in the notation; this includes $\chow$ and $\dmgm$).

\subsection{Beilinson $S$-motives 
 (after Cisinski and Deglise)}\label{sbrmot}

We list some of the  properties of the triangulated categories of Beilinson motives
(this is the version of relative
Voevodsky's motives with rational coefficients described by Cisinski and Deglise).

\begin{pr}\label{pcisdeg}

Let $X,Y$ be any (reasonable) 
schemes; $f:X\to Y$ is a (separated) finite type morphism. 

\begin{enumerate}

\item\label{imotcat} 
A tensor triangulated $\q$-linear category $\dmcx$ with the unit object $\q_X$ is defined. 

$\dmcx$ is the category of  {\it constructible Beilinson motives} over $X$, as defined (and thoroughly studied) in \S14 of \cite{degcis}.


\item\label{ivoemot} If $S$ is the spectrum of a perfect field, $\dmcs$ is isomorphic to the category $\dmgm=\dmgm(S)$ of Voevodsky's geometric motives (with rational coefficients) over $S$ (see \cite{1}). Besides, $\dmgm=\lan\chow\ra$ (here we consider the full embedding $\chow\to \dmgm$ that is a natural extension of the embedding $\chowe\to \dmge$ given by ibid.).  

\item\label{iidcompl} All 
 $\dmcx$ are idempotent complete.

\item\label{imotfun}  For any  $f$ 
the following functors
 are defined:
$f^*: \dmc(Y) \leftrightarrows \dmcx:f_*$ and $f_!: \dmcx \leftrightarrows \dmcy:f^!$; $f^*$ is left adjoint to $f_*$ and $f_!$ is left adjoint to $f^!$.

We call these the {\bf motivic image functors}.
Any of them (when $f$ varies) yields a  2-functor from the category of 
reasonable schemes
with separated morphisms of finite type to the 2-category of triangulated categories.

\item\label{iupstar}  $f^*$ is symmetric monoidal; $f^*(\q_Y)=\q_X$.

\item \label{ipur}
$f_*\cong f_!$ 
if $f$ is proper. 

If $f$ is an open immersion, we  have $f^!=f^*$. More generally, $f^!(-)\cong f^*(-)(s)[2s]$ 
 if $f$ is smooth 
 (everywhere) of relative dimension $s$.

\item \label{itr}
If $X,Y$ are regular, and $\mathcal{O}_X$ is a free finite-dimensional  $\mathcal{O}_Y$-module, then 
the adjunction morphism $M\to f_*f^*(M)$ splits for any 
 $M\in \obj \dmcy$.


\item\label{iglu}

If $i:Z\to X$ is a closed embedding, $U=X\setminus Z$, $j:U\to X$ is the complementary open immersion, then
the motivic image functors yield a  gluing data for $\dmc(-)$  (in the sense of 
Definition \ref{dglu}(1); one should set $\cu=\dmcs$, $\du=\dmc(Z)$, and $\eu=\dmcu$ in it).

 \item\label{igenc}
 $\dmcs$ (as a triangulated category) is generated by $\{ g_*(\q_X)(r)\}$, where $g:X\to S$ runs through all 
smooth separated finite type morphisms, $r\in \z$.

\item\label{ridmot} The functor $g^*$
 can be defined for any separated morphism $g$ (of schemes) not necessarily of finite type; this definition respects the composition for morphisms. 
 
Moreover, one can also define $j_K^!$ for $K\in \sss$ (see 
\S\ref{scheme}). Besides, if for composable morphisms $g,h$ (not necessarily of finite type) all of $h^!,g^!,(h\circ g)^!$ are defined (i.e. any of  $h,g,h\circ g$ is either of finite type or of the type $j_K$), then  $(h\circ g)^!\cong g^!\circ h^!$.

\item\label{icont}
Let $g:X'\to Y$ be an essentially pro-affine morphism (see Definition \ref{dmor}(2)); adopt the notation of loc. cit. Then $\dmc(X')$  is isomorphic to the $2$-colimit of the categories $\dmc(X_i)$; in these isomorphism all the connecting functors are given by the corresponding 
$(-)^*$ (cf. the previous assertion).

\item\label{iconspr} The functor $g^*$ is conservative for any 
essentially pro-affine surjective morphism $g$ (in particular, for a morphism of spectra of fields).

\item\label{iconsp} The family of functors $j_K^*$, where $K$ runs through 
$\sss$ (see \S\ref{scheme}), is conservative on $\dmcs$.

\item \label{itro}
If $g$ is a pro-finite universal homeomorphism then $g^*$ is an equivalence of categories.

\item\label{iexch} 
For a Cartesian square
of separated  
morphisms 
\begin{equation}\label{ebchs}
\begin{CD}
Y'@>{f'}>>X'\\
@VV{g'}V@VV{g}V \\
Y@>{f}>>X
\end{CD}\end{equation}
we have $g^*f_!\cong f'_!g'{}^*$ (for $g$ not necessarily of finite type) and $g'_*f'{}^!\cong f^!g_*$.

\item\label{iexchn} Adopt the notation of the previous assertion, and assume also that $g$ is a pro-finite universal homeomorphism. Then we also have
$g^*f_*\cong f'_*g'{}^*$ and $g'^*f^!\cong f'{}^!g^*$.

\item\label{iexche} We   
 have these isomorphisms also in the case when $g$ is the composition of the inverse limit of smooth affine morphisms with any smooth finite type morphism.

\item\label{il4onepoII} In the setting of assertion \ref{iglu}, 
for any $M,N\in \obj \dms$ there exists a complex
$\dmc(Z) (i^*(M),i^! (N))\to \dmcs(M,N)\to \dmc(U) (j^*M,j^* N)$ (of abelian groups) that is exact in the middle.

\end{enumerate}

\end{pr}
\begin{proof}
Most of these  statements were stated in the introduction of \cite{degcis} (and proved later in ibid.); see \S1.1 of 
 \cite{brelmot} for more details.

The first part of assertion \ref{ivoemot} 
 is given by Corollary 16.1.6 of ibid. The second part of it was proved in \S6.4 of \cite{mymot}.

Assertion \ref{itr} was established in process of the proof of Theorem 14.3.3 of \cite{degcis}. 

Assertion \ref{iconspr} is just the theorem itself (cf. Definition 2.1.7 of ibid.) in the case when $g$ is of finite type; the general case follows immediately by assertion \ref{icont}.

Assertion \ref{iconsp}  easily follows from Theorem \ref{twchow}(IV) below. 

In the case when $g$ is finite, assertion \ref{itro} is given by Proposition 2.1.9  of \cite{degcis} (note that we can apply the result cited by 
Theorem 14.3.3 of ibid.). In order to pass to the limit in this statement one should apply assertion \ref{icont} (once more).

\ref{iexchn}. Recall that for any $S$ the category $\dmcs$ is a full triangulated subcategory of a certain $\dm(S)$; $\dmcs$ {\it weakly generates} $\dms$ (i.e. $\dmcs^{\perp}=\ns$ in $\dm(S)$).
Moreover, the motivic image functors can be extended to $\dm(-)$; $g^*$ and its right adjoint $g_*$ are defined for these categories for an arbitrary morphism $g$ (of reasonable schemes). Hence in our situation $g_*$ yields an inverse isomorphism $\dmc(X')\to \dmc(X)$ (since for any $M\in \obj \dmcx$ a cone of the adjunction unit morphism $M\to g_*g^*M$ is orthogonal to $\dmcx$ by assertion \ref{itro}). Hence it suffices to verify:   $f_* g'_*\cong g_*f'_*$ and $g'_*f'{}^!\cong f^!g_*$. The first isomorphism is obvious, whereas the second isomorphism was established in \cite{degcis} for $g$ not necessarily of finite type.
 
Assertion \ref{iexche} is given by Propositions 4.3.14 and 4.3.12 of ibid., respectively.

Assertion \ref{il4onepoII} is an easy consequence of assertion \ref{iglu}.

\end{proof}

\begin{rema}
Actually, $g^*$ is conservative for any surjective morphism $g:X\to Y$ of 
reasonable schemes. Indeed, part \ref{iconsp} of the proposition reduces this statement to the case when $Y$ is the spectrum of a field;  part \ref{iconspr} allows to assume that this field is algebraically closed, and then $g$ splits. 
\end{rema}

\subsection{The Chow weight structure for $\dmcs$}\label{schow}

We define $\chows$ as the Karoubi-closure of $
\{f_*(\q_X)(r)[2r]\}$ in $\dmcs$; here $f:X\to S$ runs through all finite type projective  morphisms such that $X$ is regular, $r\in \z$.


\begin{theo}\label{twchow}

I There exists a (unique) bounded weight structure $\wchow$ for $\dmcs$ 
whose heart is $\chows$.  For any $n\in \z$ the functor $-(n)[2n]$ is weight-exact with respect to this weight structure.

II Let $f:X\to Y$ be a (separated) finite type
 morphism of 
 schemes. Then the following statements are valid.

1. $f^!$ and $f_*$ are right weight-exact; $f^*$ and $f_!$ are left weight-exact.

2. Suppose moreover that $f$ is smooth. Then $f^*$
and $f^!$ are also weight-exact.


3. $f^*$ is weight-exact also if $f$ is a finite universal homeomorphism.

4. If $f$ is proper, then $f_*\q_X\in \dmc(Y)_{\wchow\le 0}$.

III Let $K$ be a generic point of  $S$, $M\in \obj\dmcs$. 

1. Suppose that 
 $j_K^*M\in \dmck_{\wchow\ge 0}$ (resp.  $j_K^*M\in\dmck_{\wchow\le 0}$). Then there exists an open immersion $j:U\to S$, $K\in U$, such that $j^*M\in \dmc(U)_{\wchow\ge 0}$ (resp. $j^*M\in \dmc(U)_{\wchow\le 0}$).

2. Suppose that 
 $j_K^*M\in \dmck_{\wchow =0}$. Then there exists an open immersion $j:U\to S$, $K\in U$, such that $j^*M$ is a retract of 
$(g\circ h)_*\q_P(s)[2s]$, where $h:P\to U'$ is a smooth projective morphism, $U'$ is a regular scheme, $g:U'\to U$ is a finite universal homeomorphism,  $s\in \z$.

IV  $M\in \dmcs_{\wchow\ge 0}$ (resp. $M\in \dmcs_{\wchow\le 0}$) if and only if for any $K\in \sss$ we have $j_K^!(M)\in \dmc(K)_{\wchow\ge 0}$ (resp. $j_K^*(M)\in \dmc(K)_{\wchow\le 0}$).



V1. $f^*$ is left weight-exact for any morphism $f:X\to Y$ of reasonable schemes.

2. Consider  an essentially pro-affine morphism $g:X'=\prli_{i\in I} X_i\to Y$ (see Definition \ref{dmor}(2)) and the corresponding $g_i:X'\to X_i$. Then for an $M\in \obj \dmc(X')$
we have: $M\in \dmc(X')_{\wchow\le 0}$ if and only if there exist  
$i\in I$ 
and   $M_i\in \dmc(X_i)_{\wchow\le 0}$ such that $M\cong g_i^*(M_i)$.

 3. In the notation of the previous assertion, the functor $g^*$ is weight-exact if all the corresponding functors $f_i^*:\dmcs(Y)\to \dmcs(X_i)$ are so.

4. Adopt the assumptions of the previous assertion. Then for an $N\in \obj\dmc(Y)$ we have: $g^*(N)\in \dmc(X')_{\wchow\ge 0}$ (resp. $g^*(N)\in \dmc(X')_{\wchow\le 0}$)  if and only if for any large enough $i\in I$ we have $f_i^*(N)\in \dmc(X_i)_{\wchow\ge 0}$ (resp. $f_i^*(N)\in \dmc(X_i)_{\wchow\le 0}$).

5. In particular, all the parts of assertion V can be applied if $g$ is quasi-regular (see Definition \ref{dmor}(3)).



VI Assume that $g:X'\to Y$ is quasi-regular and surjective. Then  for any $ N\in \obj \dmc(Y)$, we have:  $g^*N\in \dmc(X')_{\wchow\le 0}$ (resp. $g^*N\in \dmc(X')_{\wchow\ge 0}$)
if and only if $N\in \dmc(Y)_{\wchow\le 0}$ (resp. $N\in \dmc(Y)_{\wchow\ge 0}$).

\end{theo}
\begin{proof}
Assertions I--IV were established in \cite{brelmot}; see
Theorems 2.1.2(I) and 2.2.1(I, II, III, V1), Lemma 2.2.4, 
Remark 2.3.7(4), and Proposition 2.2.3 of ibid., respectively. 
For an $f$ that is not quasi-projective assertion II was proved in (Theorem 3.7 of)
\cite{hebpo} (where assertions I-II were proved independently and somewhat earlier than in \cite{brelmot}); yet we will not actually need non-quasi-projective morphisms below.

V1. By Proposition \ref{pbw}(\ref{iwfun}) it suffices to verify that $f^*(\chow(Y))\subset \dmc(X)_{\wchow\le 0}$. 
Proposition \ref{pcisdeg}(\ref{iexch}) reduces the latter fact to assertion II4.

2. 
The "if" part is immediate from the previous assertion.

We verify the converse assertion. 
Applying Proposition \ref{pbw}(\ref{ipost}) to our setting we obtain: there exists a finite set of non-negative indices $J\subset \z$, 
certain projective $p^j:P^j\to X'$ (for all $j\in J$ and some regular $P^j$), and an $r\in \z$  such that $M$ belongs to the smallest Karoubi-closed extension-stable subclass of $\obj \dmc(X')$ containing all $p^j_*\q_{P^j}(r)[2r-j]$. 
Certainly, there exists an $i\in I$ such that all $p^j$ come via base change from certain projective $p^j_i: P^j_i\to X_i$ (for some not necessarily regular $P^j_i$; see 
Theorems 8.8.2 
and 8.10.5 of \cite{ega43}). 
 Proposition \ref{pcisdeg}(\ref{icont}, \ref{iexch}) yields: we can also assume that $M\cong g_i^*(M_i)$ for some $M_i$ belonging to  the smallest Karoubi-closed extension-stable subclass of $\obj \dmc(X_i)$ containing all $p^j_{i*}\q_{P_i^j}(r)[2r-j]$. Lastly, 
Proposition \ref{pbw}(\ref{iext}) yields that $M_i\in \dmc(X_i)_{\wchow\le 0}$.

3. We should verify the right weight-exactness of $g^*$. By Proposition \ref{pbw}(\ref{ipost}) it suffices to verify: $M\perp g^*N$ for any $M\in \dmc(X')_{\wchow\le -1}$, $N\in \dmc(Y)_{\wchow\ge 0}$. We fix $M$ and $N$; let $s\in \dmc(X')(M,g^*N)$.
By 
the previous assertions
there exist $i\in I$ and $M_i\in \dmc(X_i)_{\wchow\le -1}$ such that $M\cong g_i^*(M_i)$.
Hence 
for the transition morphisms $g_{ji}:X_j\to X_i$ (for all $j\ge i$) we have $g_{ji}^*(M_j)\in \dmc(X_j)_{\wchow \le -1}$. Next, the weight-exactness of $f_j^*$
yields that  $g_{ji}^*(M_j)\perp f^*_j(N)$. It remains to apply Proposition \ref{pcisdeg}(\ref{icont}).

4. The "$\wchow\le 0$" part of the statement is given by assertions V1--2. So, we verify the remaining part.

Consider a weight decomposition of $N[1]$: $B\stackrel{s}{\to} N[1]
{\to} A\to B[1]$. If $g^*N\in \dmc(Y)_{\wchow\ge 0}$ then $g^*(s)=0$ (since $\dmc(X')_{\wchow\le 0}\perp g^*(N)[1])$.
Applying Proposition \ref{pcisdeg}(\ref{icont}) once more we obtain:
for any large enough $i\in I$ we have $f_i^*(s)=0$. Hence $f_i^*(N)$ is a direct summand of $f_i^*(A[-1])$ for all such $i$. Since $f_i^*(A[-1])\in \dmc(X_i)_{\wchow\ge 0}$ and $\dmc(X_i)_{\wchow\ge 0}$ is Karoubi-closed in $\dmc(X_i)$, we obtain the result.
 
 Conversely, assume that $f_i^*(N)\in \dmc(X_i)_{\wchow\ge 0}$ for some $i\in I$. Then we obtain 
 $f_j^*(s)=0$. Hence $g^*(s)=0$ also; thus $g^*N$ is a retract of $g^*(A)[-1]$.

 5. The corresponding $f_i$ are weight-exact by assertions II2--3.
 
 VI The "if" statement is given by the previous assertion. It also yields: in order to verify the converse implication, it suffices to do so for  
 $g$ being either a finite universal homeomorphism or a smooth finite type surjective morphism. 
 The latter case is easy since
 $g^*$ is an isomorphism (see Proposition \ref{pcisdeg}(\ref{itro}); hence Proposition \ref{pbw}(\ref{iuni}) (together with assertion II.3) yields the result.

 Now, let $g$ be smooth surjective. First consider the case when $Y=\spe K$ ($K$ is a field; this case is the most interesting to us). 
 Since $X'$ is smooth (and hence Zariski-locally \'etale over an affine $\spe K$-space) there exists 
 a morphism $l:\spe K'\to X'$ (for a field $K'/K$) such that that the composition 
  (structure) morphism $k'=g\circ l:\spe K'\to \spe K$ is \'etale. 

If $g^*N\in \dmc(X')_{\wchow\le 0}$,  then $k'^*N=l^*(g^*N)\in \dmc(K')_{\wchow\le 0}$ (see 
assertion II.1). Since $k'_*$ is weight-exact, we also have $k'_*k'^*N\in \dmck_{\wchow\le 0}$. 
Since $N$ is a retract of  $b'_*b'^*N$ (see Proposition \ref{pcisdeg}(\ref{itr})), we conclude that $N\in \dmck_{\wchow\le 0}$.

Next, if $g^*N\in \dmcs_{\wchow\ge 0}$, then 
we also have  $g^!N\in \dmcx_{\wchow\ge 0}$ (by assertion I and Proposition \ref{pcisdeg}(\ref{ipur})).
Since $k'^*N=k'^!N$, we obtain $k'^*N\in \dmc(\spe K')_{\wchow\ge 0}$. Hence  $k'_*k'^*N\in \dmc(K')_{\wchow\ge 0}$
and Proposition \ref{pcisdeg}(\ref{itr}) yields the result (again).

 In order to reduce our claim from the case of arbitrary finite type smooth morphisms to the one when $Y$ is the spectrum of a field, we apply assertion IV. Here we apply the $2$-functoriality of $(-)^*$ when we consider the case $g^*N\in \dmc(X')_{\wchow\le 0}$, and combine it with Proposition \ref{pcisdeg}(\ref{ipur}) when $g^*N\in \dmc(X')_{\wchow\ge 0}$.

\end{proof}

\begin{rema}\label{runrwchow}
Note also that an alternative construction of $\wchow$  over any (not necessarily reasonable) excellent separated finite-dimensional  scheme $S$ was considered in  \S2.3 of \cite{brelmot}.
Its functoriality properties were only studied with respect to quasi-projective morphisms; yet this is quite sufficient for our purposes.
\end{rema}

Below we will call weight spectral sequences and weight filtrations corresponding to $\wchow$ the {\it Chow-weight} ones.

\subsection{On the \'etale realization of motives and the perverse $t$-structure} \label{setre}

This paragraph is dedicated to the justification of the following statement.

\begin{theo}\label{tetre}
Let $S/\spe\zol$ be a reasonable scheme.

I.1. We have an exact  
functor $\hetl(S):\dmcs\to D^b_cSh^{et}(S,\ql)$; the latter is the triangulated category of constructible \'etale $\ql$-sheaves over $S$ (see below).

2. The target categories $D^b_cSh^{et}(-,\ql)$ of $\hetl(-)$ are equipped with connecting functors of the type $g^*$ for  $g$ being any morphism of reasonable schemes, and also with $f_*$, $ f_!$, and $f^!$ for $f$ being a finite type morphism.

Moreover, $\hetl(-)$ converts the corresponding motivic image functors (see Proposition \ref{pcisdeg}(\ref{imotfun}, \ref{ridmot})) into these \'etale versions. 



II The category $D^b_cSh^{et}(-,\ql)$ (for any reasonable $S/\spe\zol$) is equipped with 
a bounded perverse $t$-structure (for the self-dual perversity) that we will denote just by $t$; the heart of $t$ will be denoted by $\shs$.
The collection of these $t$-structures 
over all reasonable schemes
enjoys the following  properties (for $f:X\to Y$ being a finite type morphism of reasonable schemes).

1. If $f$ is an immersion, then $f_*$ and $f^!$ are left $t$-exact, whereas  $f_!$ and $f^*$ are right $t$-exact (see Definition \ref{dmts}(3)).

2. If $f$ is affine, then $f_!$ is left $t$-exact, and  $f_*$ is right $t$-exact.

3. If $f$ is quasi-finite affine, then $f_*$ and $f_!$ are $t$-exact. 

4. If $f$ is proper of relative dimension $\le d$, then $ f_*[d](=f_![d])$ is  left $t$-exact, and  $f_*[-d]$ is right $t$-exact.

5. If $f$ is smooth (everywhere) of dimension $d$, then $f^![-d]$ and $f^*[d]$ are $t$-exact. 

6. If $K$ is  
a point of $S$  
of dimension $d$ (see \S\ref{scheme}), 
then $j_{K}^*[-d]$ (resp. $j_{K}^![-d]$; see Remark \ref{rjk} below) is left (resp. right) $t$-exact. 

7. Moreover, for $M\in \obj D^b_cSh^{et}(S,\ql)$ we have $M\in D^b_cSh^{et}(S,\ql)^{t\le 0}$ (resp. $M\in D^b_cSh^{et}(S,\ql)^{t\ge 0}$) whenever
for any $K\in \sss$, $K$ is of 
dimension $d$, we have $j_{K}^*M[-d]\in D^b_cSh^{et}(K,\ql)^{t\le 0}$ (resp. $j_{K}^!M[-d]\in D^b_cSh^{et}(K,\ql)^{t\le 0}$). 

8.  For  a closed  embedding $i:Z\to S$ 
and the complementary immersion $j:U\to S$ for $M\in \obj D^b_cSh^{et}(S,\ql)$ we have: $M\in D^b_cSh^{et}(S,\ql)^{t\le 0}$ (resp. $M\in D^b_cSh^{et}(S,\ql)^{t\ge 0}$) whenever $j^*M\in D^b_cSh^{et}(U,\ql)^{t\le 0}$ and  
$i^*M\in D^b_cSh^{et}(Z,\ql)^{t\le 0}$ (resp. $j^*M\in D^b_cSh^{et}(U,\ql)^{t\ge 0}$ and  
$i^!M\in D^b_cSh^{et}(Z,\ql)^{t\ge 0}$).

III Let $g:X\to Y$ be a surjective morphism of (reasonable) schemes.
Then the following statements are valid.

1. $g^*$ is conservative.

2. If $g$ is smooth 
 (everywhere) of dimension $d$, then for $M\in \obj D^b_cSh^{et}(Y,\ql)$ we have: $M\in \obj D^b_cSh^{et}(Y,\ql)^{t\le 0}$ (resp.  $M\in \obj D^b_cSh^{et}(Y,\ql)^{t\ge 0}$) if and only if  $g^*[d](M)\in \obj D^b_cSh^{et}(X,\ql)^{t\le 0}$ (resp. $g^*[d](M)\in \obj D^b_cSh^{et}(X,\ql)^{t\ge 0}$).

IV Let 
 $g:X\to Y$ be a
 morphism such that for any irreducible component $Y_i$ of $Y$ the dimension of $g\ob(Y_i)$ equals $\dim Y_i+d$ (for some fixed $d\in \z$). 
Then the following statements are valid.

1.  $g^*[d]D^b_cSh^{et}(X,\ql)^{t\le 0}\subset D^b_cSh^{et}(X,\ql)^{t\le 0}$.

2. Moreover, $g^*$ is $t$-exact if $g$ is a universal homeomorphism.

3. Assume that $g$ is a projective limit of smooth affine morphisms.
Then $g^*[d]$ is $t$-exact.

V For $S$ of finite type over $\sfp$ ($p\neq l$), $D^b_cSh^{et}(S,\ql)$ is the category of constructible 
$\ql$-perverse sheaves as defined in \cite{bbd}, 
and $t$ is the corresponding perverse $t$-structure for the self-dual perversity.

\end{theo}
\begin{proof}
I.1. 
Theorem 16.2 of \cite{degcis} yields that $\dmcs$ is isomorphic to a certain triangulated category $DM_{h,c}(S)$ (of constructible $h$-motives). Moreover, Theorem 16.1.3 of ibid. yields that this isomorphism is compatible with all the functors required. 
Next, Theorem 5.9.21 of \cite{cdet} together with the discussion in \S5.9.19 yields the existence of the realization functor in question.

2. The compatibility of $\hetl(-)$ with the motivic image functors for finite type morphisms (see Proposition \ref{pcisdeg}(\ref{imotfun})) is given by loc. cit. The compatibility with $g^*$ for an arbitrary morphisms $g$ of reasonable schemes 
is immediate from the fact that the diagram (5.9.10a) of ibid. is a {\it premotivic adjunction}.

II It is stated in Theorem 6.3 of \cite{eke} that the results of \cite{bbd} carry over to $D^b_cSh^{et}(-,\ql)$ (over schemes that are called reasonable in this paper). Moreover, the properties II.1--7 of $t$ 
in a closely related setting are listed in \S2.6 of \cite{huper}.

8. By Theorem 6.3(iv) of \cite{eke},  the triangulated categories along with the  connecting functors mentioned yield a gluing data. 
By Proposition \ref{pglu}(II2) it remains to note that 
$i_*$ and $j^*$ are $t$-exact. 
 
III.1. The statement easily follows from its 
$\zlz$-coefficients analogue given by Proposition 9.1 of \cite{sga48}.
Note here: by definition, the categories $D^b_cSh^{et}(-,\ql)$ are defined as the rational hulls of certain $D^b_cSh^{et}(-,\zl)$, whereas the latter are equipped with conservative functors to $D^b_cSh^{et}(-,\zlz)$ (see Theorem 6.3(i) of \cite{eke} and \S2.2 of \cite{bbd}). Moreover, both of these "change of coefficient" functors are compatible with all the connecting functors of assertion I.2. 

2. Immediate from the previous assertion together with the $t$-exactness of $g^*[d]$ (see assertion I5). Note here that a $t$-exact conservative functor cannot kill non-zero objects in the heart of the corresponding $t$; hence we can apply Remark \ref{rts}(4).

IV1. Obvious from assertion II6.

2. We should verify the left $t$-exactness of $g^*$. Now, $g^*$ 
(in this case) possesses an inverse functor $g_*$ (see Theorem 1.1 of \cite{sga48}; note that it is sufficient to know this statement for
 complexes of $\zlz$-module sheaves). Hence it suffices to note that the functors $-_*$ possesses the base change property with respect to $-^!$ (by 
Corollary 3.1.12.3 of \cite{sga418}): for any cartesian diagram
\begin{equation}\label{ebch} \begin{CD} 
 x @>{i_X}>> X\\
@VV{g'}V @VV{g}V\\
y @>{i_Y}>> Y
\end{CD}\end{equation} 
we have $g'_*\circ i_{X}^!\cong i_{Y}^!\circ g_*$ (note again that it suffices to verify the latter fact for the categories $D^b_cSh^{et}(-,\zlz)$). 

3. Again, we should verify the  left $t$-exactness property.
 Let $y$ be a Zariski point of $Y$; consider the diagram (\ref{ebch}) again.
 Then we have $g'^*\circ i_{Y}^!\cong i_{X}^!\circ g^*$; this statement can be proved via reducing to the case of $\zlz$-coefficients and applying the argument used in the proof of Proposition 4.3.12 of \cite{degcis}.
 This reduces the statement to the case when $Y$ is (the spectrum of) a field. In the latter case it is an easy consequence of absolute purity (see Theorem XVI.3.1.1 of \cite{illgabb}).

V This is just a partial case of the definition of $\dshl$.


\end{proof}

\begin{rema}\label{rjk}
1. For a $K\in \sss$ (see \S\ref{scheme}) one can define the \'etale version of $j_K^!$ using the 'classical' method (see \S2.2.12 of \cite{bbd}): for its decomposition $K \stackrel{i_K}{\to}\overline{K}\stackrel{j_{\overline{K}}}{\to} S$ ($\overline{K}$ is the closure of $K$ in $S$) one should take $j_K^!=i_K^*\circ j_{\overline{K}}^!$. This definition is certainly compatible with its motivic version (see Remark 1.1.3(2) of \cite{brelmot}). 

2. Since no proof is provided in (the current version of) \S5.9.19 of \cite{cdet} for the existence of the \'etale realization of $S$-motives with values in the 
Ekedahl's version of  $D^b_cSh^{et}(S,\ql)$
 (for a general reasonable $S$), the author will describe certain alternative methods for establishing  Theorem \ref{tetre}.
 
 Firstly, one may restrict it to the case when the $l$-adic cohomological dimension of all the schemes in question is bounded (note that this a rather mild restriction in the case $l\neq 2$). Then one may proceed via applying Theorem 9.7 of \cite{ayet} (together with the results of \S16.1 of \cite{degcis} mentioned above; cf. also \S6 of \cite{kahnconj}). 

Another possibility (that allows to avoid any restrictions and possibly even generalize our results to non-reasonable schemes) is to consider the "h-version" of the realization functor; this is the version that was actually constructed in (Theorem 5.9.21 of) \cite{cdet} (for a general excellent Noetherian separated finite-dimensional $S$). Though the author does not know how to prove that the target category 
$D^b_c(S,\ql)^\natural$ of this functor is isomorphic to $D^b_cSh^{et}(S,\ql)$, there still exist the corresponding analogue of  $D^b_cSh^{et}(S,\zl)$ equipped with the exact functor to ("the usual") $D^b_cSh^{et}(S,\zlz)$ (see Corollary 5.4.6 of ibid.). Moreover,  these categories and the natural functors between them possess all the properties of $l$-adic sheaves that were established in \cite{eke}. So, one can assume that Theorem \ref{tetre} is fulfilled for this "modified" realization.

For both of these possibilities the category $D^b_cSh^{et}(S,\ql)$ coincides with the one considered in \cite{bbd} if $S$ is of finite type over $\sfp$ (see Proposition 5.9.18 of \cite{cdet}); this is very important for the weight arguments of the next paragraph.

3. 
Note that we will not really need parts II2--3 of our theorem below, and we will only need part II4 in the (trivial) case of finite extensions of fields.
\end{rema}

\subsection {On weights for 
perverse $S$-sheaves; the degeneration of Chow-weight spectral sequences for \'etale homology}\label{swsshetl}

Below we will need certain 'weights' for $\hetl(-)$. 
First we recall that in certain cases weights are defined on $\dhsl$. 

\begin{pr}\label{phuwe}

I Let $S$ be a finite type (separated) $\sfp$-scheme (for a prime $p\neq l$).

Then there exist certain $\dhsl_{w\le 0},\ \dhsl_{w\ge 0}\subset \obj \dhsl$,  that satisfy the following properties.

\begin{enumerate}

\item\label{itrus} For any $m\in \z$ denote $\dhsl_{w\le 0}[m]$ by $\dhsl_{w\le m}$ and  denote 
$\dhsl_{w\ge 0}[m]$ by $\dhsl_{w\ge m}$. 

Then for $X\in \obj \dhsl$ we have: $X\in \dhsl_{w\le m}$ (resp. $X\in \dhsl_{w\ge m}$) if and only if for any $j\in \z$ we have $H^t_j(X)\in \dhsl_{w\le m+j}$ (resp. $H^t_j(X)\in \dhsl_{w\ge m+j}$).

\item\label{iperp} 
 Denote $\dhsl_{w\le m}\cap \dhsl^{t=0}$ by $\shs_{w\le m}$,  and  denote 
$\dhsl_{w\ge m}\cap \dhsl^{t=0}$ by $\shs_{w\ge m}$; $\shs_{w=m}=\shs_{w\le m}\cap \shs_{w\ge m}$. Then  $\shs_{w=m}$ yield exact abelian subcategories of $\shs$ that contain all $\shs$-subquotients of their objects. Besides,
for $j\neq m\in \z$ we have $\shs_{w=m}\perp \shs_{w=j}$. 

\item\label{inti}
For any open immersion $j:U\to S$ and $m\in \z$  (the perverse sheaf version of) the functor $j_{!*}$ sends $\shs(U)_{w\le m}$ into $\shs_{w\le m}$, and   sends $\shs(U)_{w\ge m}$ into $\shs_{w\ge m}$.

\item\label{ihuw} $\hetl$ is 'weight-exact' i.e. it sends $\dmcs_{\wchow\le m}$ into  $\dhsl_{w\le m}$, and sends $\dmcs_{\wchow\ge  m}$ into  $\dhsl_{w\ge m}$. 

\end{enumerate}

II Let  $S$ be a finite type (separated) $\sq$-scheme. Present $S$ as an inverse limit of finite type $\szol$-schemes $S_i$ (with connecting morphisms being open embeddings), and define $\hetlw(S)$ as the direct limit of $\hetl(S_i)$ (its target $\dhslw$ is the $2$-colimit of the corresponding $\dhsli$). Then $\hetl(S)$ can be factored through $\hetlw(S)$. Moreover, $\dhslw$ possesses a (perverse) $t$-structure that is compatible with $t$ (with respect to this connecting functor). Lastly, for $\dhslw$ one can define weights such that the analogues of all of the assertions of part I are fulfilled.

\end{pr}
\begin{proof}

I  All  the assertions except (\ref{ihuw}) are well-known properties of weights of mixed complexes of sheaves that were established in \S5 of \cite{bbd}, 
whereas assertion \ref{ihuw} was verified in \S3.6 of \cite{brelmot}.

II 
The $t$-exactness of the connection functor $\dhslw\to\dhsl$ is immediate from Theorem \ref{tetre}(IV3). 
Everything else was verified in \S3 of \cite{huper}, except the analogue of assertion 
I.\ref{ihuw} that was established in \S3.4 of \cite{brelmot} (in this case). 

\end{proof}

\begin{rema}\label{lrnowsmcsh}

1. For $S$ being of finite type over $\sfp$ 
 there exists a weight filtration (in the sense of Definition \ref{dwfilt}) for the heart of  the perverse $t$-structure for the category $\dbm(X,\ql)$ of mixed complexes of sheaves; see Theorem 5.3.5 of \cite{bbd}.

 On the other hand, the corresponding 'pure factors' $\aum$ are not semisimple. Indeed, there are non-trivial $1$-extensions of pure (perverse) sheaves even in the case when $S=\sfp$ (since for 'abstract' pure Galois representations of $\gal(\fp)$ the action of the Frobenius does not have to be semisimple due to the fact that one cannot impose any polarizability restrictions on these representations). 
 Since  all $\cu$-extensions 
 in the heart of a weight structure for $\cu$ necessarily split (immediately from the orthogonality axiom), this weight filtration does not yield a weight structure for this category.  

The author made an attempt to 'axiomatize' this setting by introducing the notion of a {\it relative weight structure}; see \S3.5--3.6 of \cite{brelmot}.


2. For $S$ being of finite type over $\sq$ the 
category of mixed perverse sheaves $\dhslw$ does not possess a weight filtration (in our sense) at all; cf. the Warning preceding Proposition 3.4 of \cite{huper}.
 The problem here is that (due to the non-vanishing of the corresponding extension groups) one can construct a mixed perverse sheaf whose pure factors are 'in the wrong order'. 
Besides, note that $\dhslw$ is not isomorphic to $\dhsl$; see the 
end of \S1 of ibid.

3. So, in both of these cases we do not have (a 'true') weight structure for 
(any sort of) triangulated category of complexes of mixed $S$-sheaves; hence the properties of (relative) motives are somewhat better than the ones of mixed complexes of sheaves (even over a finitely generated base). 

\end{rema}


Now we prove the main properties of Chow-weight spectral sequences for $\hetlz$.
To this end  we state the following conjecture.

\begin{conj}\label{cdeg}
The spectral sequence $T_{\wchow}(\hetlz,M)$ degenerates at $E_2$ for any $M\in \obj \dmcs$ (for any reasonable $S$).
\end{conj}

\begin{theo}\label{tdeg}

I.1.  Let $S$ be a finite type separated scheme over $\sfp$, $M\in \obj\dmcs$, $H=\hetlz$. Then $E_s^{pq}T_{\wchow}(H,M)\in \shs_{w=q}$ 
 for any $p,q\in\z$, $s>0$.   
Besides, $(W_mH)(M)$ is a  filtration of $H(M)$ whose $m$-th factor belongs to $\shs_{w=m}$ (for all $m\in \z$). Moreover, this filtration is uniquely and functorially characterized by the latter property.

2. For $S$ being separated of finite type over $\sq$ and  $H=\hetlwz$ the (obvious) analogue of assertion I.1 is fulfilled. 

II Let $S$ be a characteristic $p$ 
very reasonable scheme  ($p\neq l$; it can be $0$).

Then the following statements are valid.

1. Conjecture \ref{cdeg} holds. 

2. Let $j:U\to S$ be an open 
embedding; denote the complementary closed embedding by $i$. For $M\in \dmcu_{\wchow\ge s}$ (resp. $M\in \dmcu_{\wchow\le s}$)  suppose that $\hetlm(M)=0$ for all $m\neq 0$.  

 Then 
$(W_{s+m+1}\hetlm)(i^!j_!M)= 0$ for any $m >0$  (resp. $(W_{s+m}\hetlm)(i^*j_{*} M)=\hetlm(i^*j_{*} M)$ for any $m< 0$). 
\end{theo}
\begin{proof}

I.1. Proposition \ref{phuwe}(I.\ref{ihuw}) yields that $E_1^{pq}T_{\wchow}(\hetlz,M)\in \shs_{w=q}$. 

Hence the same is true for $E_s^{pq}(T)$ for any $s\ge 1$ (since
$E_s^{pq}$ is a subfactor of $E_1^{pq}(T)$; here we apply Proposition \ref{phuwe}(I.\ref{iperp})). 

Hence 
$E_{\infty}^{p+q}(T)\in \shs_{w=q}$ also, and we obtain that the factors of the Chow-weight filtration are of the 
weights prescribed. Now, the orthogonality of (subquotients of) perverse sheaves of distinct weight yields that 
this condition determines the filtration in a functorial way. 

The same argument proves assertion I2.

II.1. We verify the degeneration of $T_{\wchow}(\hetlz,M)$ for some fixed $M\in \obj \dmcs$ via reducing it to the case when $S$ is of finite type over the corresponding prime field (in three steps).
In each of these reduction steps we will apply Remark \ref{rwss}(3). 

Let $f:S'\to S$ be a smooth surjective morphism as in Definition \ref{dmor}(4). Note: we can assume that $f$ is equidimensional and that all the connected components of $S'$ are irreducible.
Now, it suffices to verify the degeneration for 
$T_{\wchow(S')}(\hetlz(S'),f^*M)$ instead of $T_{\wchow(S)}(\hetlz(S),M)$.
Indeed, the \'etale $f^*$ is $t$-exact (up to a shift) and conservative (see Theorem \ref{tetre}(II.5, III.1)) whereas the motivic $f^*$ is $\wchow$-weight exact (see Theorem \ref{twchow}(II2)); hence we can construct a diagram of the type (\ref{edeg}) and apply Remark \ref{rwss}(3) (for the first time). Moreover, it suffices to verify the degeneration of $T_{\wchow}$ for the restrictions of $f^*M$ to all of the connected components of $S'$. For this reason we will assume below that $S'$ is irreducible.
 
Next for $S'=\inli S'_i$ 
 we apply Proposition \ref{pcisdeg}(\ref{icont}) in order to 
find an index  $i$ such that $f^*M=g^*M_i$
for the corresponding $g:S'\to S'_i$ and some $M_i\in \obj\dmc(S'_i)$.
  Since the \'etale $g^*$ is $t$-exact up to a shift (see Theorem \ref{tetre}(IV3)) and the motivic one is weight-exact (see Theorem \ref{twchow}(V5)), we reduce the statement to the degeneration of $T_{\wchow(S'_i)}(\hetlz(S'_i), M_i)$.

 Next, since $S_i'$ is of finite type over a field, there exists a finite extension of $\sfp$ or $\sq$ (in the case $p=0$) such that $S_i'$ is defined over it.
 Hence (by Proposition \ref{pcisdeg}(\ref{icont})) there exists a morphism $h:S'_i\to S''$ that satisfies the following properties: $S''$ is of finite type over the corresponding prime field, $M'=h^* M''$ for some $M''\in\obj\dmc(S'')$, and $h$ is the composition of a pro-finite 
 universal homeomorphism with  a projective limit of smooth affine morphisms. Hence (by the same arguments as above; see also Theorem \ref{tetre}(IV2)) it suffices to verify the statement for 
$T_{\wchow(S'')}(\hetlz(S''),M'')$

Now consider the case $p>0$ (i.e.
$S'$ is of finite type over $\sfp$).
In this case assertion I.1 
yields that $E_s^{pq}T_{\wchow(S'')}(\hetlz(S''),M'')\in Sh_{per}^{et}(S'')_{w=q}$ (for any $p,q\in \z$, $s>0$). 
Now (by Proposition \ref{phuwe}(I.\ref{iperp})) there are no non-zero morphisms between 
distinct  $Sh_{per}^{et}(S'')_{w=m}$. Hence 
all the connecting morphisms for $E_s(T)$ vanish for all $s>1$, and we obtain the result.

In the case when $S$ is of finite type over $\sq$ we note that the same argument proves the degeneration of $T_{\wchow(S'')}({\tilde H}^{et}_{\ql,0}(S''),M'')$; hence the functoriality of Chow-weight spectral sequences (with respect to $H$; see Proposition \ref{pwss}(I)) yields the assertion desired.


2. The same reduction arguments as above (along with Proposition \ref{pcisdeg}(\ref{iexch}--\ref{iexche})
 enable us to assume that  $S$ is of finite type over $\sfp$ (for $p\neq 0$) or over $\sq$. 
In this case 
Theorem \ref{tetre}(I,II) along with assertion I allow us to translate
our assertion into the corresponding analogue for weights on 
$\dhsl$. 
Now, for $S/\sfp$ we have
$i^!\dhsl_{w\le s}\subset \dhzl_{w\le s}$ and $i^*\dhsl_{w\ge s}\subset \dhzl_{w\ge s}$ by (\S5.1.14, (i*) and (i), of) \cite{bbd}.   Proposition \ref{pglu}(IV\ref{iim}) yields:  it suffices to note that $j_{!*}$ respects weights of mixed 
 sheaves; this is Corollary 5.3.2 of \cite{bbd}. 

For $S/\sq$ it suffices to verify the assertion for $\hetlmw$ instead of $\hetlm$. In this setting we can apply the Remark succeeding Definition 3.3 
and  Corollary 3.5 of \cite{huper} (instead of  the results of \cite{bbd} cited above).


\end{proof}

\begin{rema}\label{rdeg} 
1. Using Verdier duality (for motives or sheaves) one can easily 
carry over 
the results above from \'etale homology to \'etale cohomology. 

2. In the characteristic $0$ case of Theorem \ref{tdeg}(II), we could have tried to use M. Saito's Hodge modules in our weight arguments (in order to avoid the usage of $\hetlw$). The main problem here is that (to the knowledge of the author) no 'Hodge module realization' of motives is known to exist at the moment (still see the proof of Proposition 7.6 of \cite{wildiat} for a certain reasoning avoiding this difficulty). 

Alternatively, one 
could try to reduce the characteristic $0$ case to the positive characteristic one using the methods and results of \S6 of \cite{bbd}. 

\end{rema}

\section{On the existence of a (nice) motivic $t$-structure}\label{smain}

In \S\ref{sdmts} we define a (motivic) $t$-structure $t_l$ for $\dmcs$ as the one that is strictly compatible with the perverse $t$-structure for the $\ql$-\'etale homology (cf. \S2.10 of \cite{beilnmot}). We also study the functoriality of this definition. 

In \S\ref{spmts} 
we reduce the existence of $t_l$ to the case when $S$ is the (spectrum of) a universal domain (of characteristic distinct from $l$). Moreover, the existence of $t_l$ over universal domains automatically yields  that Chow-weight filtrations and Chow-weight spectral sequences can be lifted from $\shs$  to motives. When $S$ is a very reasonable scheme, the weight filtration for $\hrtl$ obtained this way is strictly compatible with morphisms. 

In \S\ref{scons}
we study 
certain properties of motives that follow from the {\it niceness} of $t_l$ (i.e. from its transversality with $\wchow$).

In \S\ref{snice} we apply these results (in a certain Noetherian induction step).
We prove that a nice $t_l$ 
exists over an arbitrary very reasonable scheme $S$ of characteristic $p$ if 
such a $t_l$ exists over some universal domain of the same characteristic. 


\subsection{The motivic $t$-structure (for 
\texorpdfstring{$S/\szol$}{S/zol})} \label{sdmts}

Till \S\ref{sindl} we will fix some prime $l$, and will usually 
assume that all the schemes we consider are $\szol$-ones. In this case we will define the motivic $t$-structure in terms of $\hetl$; we will treat the question whether it actually depends on $l$ later.

\begin{defi}\label{dmts}

Let $S$ be a (reasonable) scheme. 

1. Consider the 
class $\dmcs^{t_l\le 0}$ (resp. $\dmcs^{t_l\ge 0}$) consisting of those $M\in \obj\dmcs$
that satisfy:  $\hetl(M)\in \dshsl^{t\le 0}$ (resp. $\hetl(M)\in \dshsl^{t\ge 0}$; see Theorem \ref{tetre}).

2. For a $\szol$-scheme $S$ if $(\dmcs^{t_l\le 0},\ \dmcs^{t_l\ge 0})$ yield a 
$t$-structure for $\dmcs$, we will say that (the $t$-structure) $t_l$ exists for $\dmcs$, or that it exists over $S$. We will denote the heart of $t_l$ (in this case) by $\mm(S)$. 

 
 3. We will use the term "(left, right, or both) $t$-exact functor" for functors between certain
 $\dmc(-)$ that respect (the 'halves of') $t_l$ in the corresponding way without (necessarily) assuming that $t_l$ yields a $t$-structure. 
 

4. If 
$t_l$ exists for $\dmcs$, we will say that it is {\it nice} if
it is transversal to $\wchow$.

\end{defi}

\begin{rema}\label{rbound}

If 
$t_l$ exists over $S$, then it is automatically bounded, since the \'etale homology of any object of $\dmcs$ is. The latter fact is immediate from  Proposition \ref{pcisdeg}(\ref{igenc}) and Theorem  \ref{tetre}(I,II).

In particular, we obtain that $t_l$ is non-degenerate (see Remark \ref{rts}(4)) and that $\hetl$ is conservative. Note that the existence of $t_l$ is a very strong assumption!

\end{rema}

We will need some functoriality properties of $(\dmc(-)^{t_l\le 0},\ \dmc(-)^{t_l\ge 0})$ below; certainly, they become even more interesting (for themselves) if $t_l$ exists.

\begin{lem}\label{lfumts}
Let $f:X\to Y$ be a morphism of 
schemes. Then the following statements are valid.

1. If $f$ is an immersion, then $f_*$ and $f^!$ are left $t$-exact, whereas  $f_!$ and $f^*$ are right $t$-exact (see Definition \ref{dmts}(3)).

2. If $f$ is affine, then $f_!$ is left $t$-exact, and  $f_*$ is right $t$-exact.

3. If $f$ is quasi-finite affine, then $f_*$ and $f_!$ are $t$-exact. 

4. If $f$ is proper of relative dimension $\le d$, then $ f_*[d](=f_![d])$ is  left $t$-exact, and  $f_*[-d]$ is right $t$-exact.

5. If $f$ is smooth of dimension $d$, then $f^![-d]$ and $f^*[d]$ are $t$-exact. 


6. If $K$ is  
a point of $S$ 
of dimension $d$ (see \S\ref{scheme}), 
then $j_{K}^*[-d]$ (resp. $j_{K}^![-d]$) is left (resp. right) $t$-exact. 


7. Moreover, for $M\in \obj \dmcs$ we have $M\in \dmcs^{t_l\le 0}$ (resp. $M\in \dmcs^{t_l\ge 0}$) if and only if
for any $K\in \sss$, $K$ is of 
dimension $d$, we have $j_{K}^*M[-d]\in \dmck^{t_l\le 0}$ (resp. $j_{K}^!M[-d]\in \dmck^{t_l\ge 0}$). 


8.  For  a closed  embedding $i:Z\to S$ 
and the complementary immersion $j:U\to S$ for $M\in \obj \dmcs$ we have: $M\in \dmcs^{t_l\le 0}$ (resp. $M\in \dmcs^{t_l\ge 0}$) if and only if $j^*M\in \dmcu^{t_l\le 0}$ and  
$i^*M\in \dmc(Z)^{t_l\le 0}$ (resp. $j^*M\in \dmcu^{t_l\ge 0}$ and  
$i^!M\in \dmc(Z)^{t_l\ge 0}$). %

\end{lem}
\begin{proof}
Immediate from Theorem \ref{tetre}(II).

\end{proof}

Now we formulate the first of the main results of this paper.

\begin{theo}\label{tmts}

Suppose that for any 
 point $K$ of (a reasonable $\szol$-) scheme $S$ there exists $t_l$ for
 $\dmck$.
 
Then 
$t_l$ exists for $\dmcs$ 
also. 
\end{theo}

\subsection{The proof of the 'globalization' theorem for $t_l$}\label{spmts}

Till \S\ref{sindl} we will assume that $S$ is a (reasonable) $\szol$-scheme.

We will need the following statement.

\begin{lem}\label{lcont}

Let $K$ be a generic point  of a scheme $U'$ whose dimension is $d$ (see \S\ref{scheme}); denote the morphism $K\to U'$ by $j_K$.

Let $M\in \obj\dmc(U')$, and suppose that 
 $j_K^*M[-d]\in \dmck^{t_l\le 0}$ (resp.  $j_K^*M[-d]\in\dmck^{t_l\ge 0}$). 
 Then there exists an open immersion $j:U''\to U'$, $K\in U''$, such that $j^*M\in \dmc(U'')^{t_l\le 0}$ (resp. $j^*M\in \dmc(U'')^{t_l\ge 0}$).

\end{lem}
\begin{proof}

Theorem \ref{tetre}(I,II)  reduces this fact to its $\dshs$-version. 
Applying Verdier duality, we obtain that it suffices to verify the following statement: for any $C\in \obj \dhupl$ if $j_K^*(C)[-d]\in \dshkl^{t\le 0}$, then there exists an open immersion $j:U''\to U'$, $K\in U'$, such that $j^*C\in \dhuppl^{t\le 0}$. 
Now, over $K$ the perverse $t$-structure for $\dshkl$ coincides with the 'canonical' one (corresponding to the canonical $t$-structure for the derived category $\dshkl$), whereas over any $U''$ any ('ordinary') constructible $\ql$-sheaf belongs to $\dhuppl^{t\le 0}$. 
Considering the canonical homology of $C$ (note that $j^*$ and $j_K^*$ are exact when restricted to the category of 'ordinary' $\ql$-sheaves) we obtain that it suffices to verify: if the stalk of some  constructible $\ql$-sheaf $T$  at $K$ is zero, then 
for some open $U''\subset U'$, $K\in U''$, we have $j^*T=0$. This is immediate from 
Proposition I.12.10 of \cite{frki}. Note here: one can apply the method of the proof of loc. cit. in our (more general) setting by Theorem 6.3(i) of \cite{eke}.



\end{proof}

Now we 
prove Theorem \ref{tmts}.

We should prove that 
$(\dmcs^{t_l\le 0}, \dmcs^{t_l\ge 0})$ (see Definition \ref{dmts}(1)) yield a $t$-structure 
for $\dmcs$.

Obviously, to this end it suffices to verify that for $\dmcs^{t_l\le 0}$ and $\dmcs^{t_l\ge 0}$ prescribed by Definition \ref{dmts} we have the orthogonality property, and that $t_l$-decompositions exist.

The proof of orthogonality 
uses 
an argument contained in the proof of Proposition 2.2.3 of 
\cite{brelmot}. 
 We apply Noetherian induction. Suppose that the assertion is fulfilled over any (proper) closed subscheme of $S$.

For any (fixed) $M\in \dmcs^{t_l\le 0}$, $N\in \dmcs^{t_l\ge 1}$, $h\in \dmcs(M,N)$,
we should prove that  $h=0$.

Let $K$  be a generic point of $S$ of dimension $d$. 
Lemma \ref{lfumts}(6) yields that $j_K^*M[-d]\in   \dmck^{t_l\le 0}$, $j_K^!N[-d]\in \dmck^{t_l\ge 1}$.
Hence $j_K^*h=0$ (since 
$t_l$ exists for $K$-motives). Hence (by Proposition \ref{pcisdeg}(\ref{icont})) 
there exists an open immersion $j:U\to S$, $K\in U$,   such that  $j^*h=0$. Let $i:Z\to S$ denote the complementary 
closed embedding; Lemma \ref{lfumts}(1) yields that $i^*(M)\in   \dmc(Z)^{t_l\le 0}$, $i^!N\in \dmc(Z)^{t_l\ge 1}$. 
By the inductive assumption (applied 
to $Z$) we have $\dmc(Z)(i^*(M),i^!(N))=\ns$. Hence Proposition \ref{pcisdeg}(\ref{il4onepoII}) yields the assertion.
 
It remains to verify the existence of a $t_l$-decomposition for an $M\in \obj \dmcs$. We use the method of the proof similar to that of Proposition 2.3.4 of \cite{brelmot}. Again, we apply Noetherian induction and assume that the assertion is fulfilled over any proper closed subscheme of $S$.

We choose some generic point $K$ of $S$. We consider the $t_l$-decomposition 
\begin{equation}\label{tdk}
A_K[-d]\to j_K^*M[-d]\to B_K[-d]
\end{equation}
 (of $j_K^*M[-d]$ in $\dmck$). We verify that there exists an open immersion $j:U\to S$ containing $S$  
such that  (\ref{tdk}) (shifted by $[d]$) lifts to a $t_l$-decomposition of $j^*M$. 
By Proposition \ref{pcisdeg}(\ref{icont}) it suffices to verify: for any open $U'\subset S$ containing $K$ and any $A_{U'},B_{U'}\in \dmc(U')$  such that the 'restriction' of $(A_{U'},B_{U'})$ to $K$ equals $(A_K,B_K)$,  there exists an open $U\subset U'$ (containing $K$) such that 'restrictions' $A_U,B_U$ of $A_{U'},B_{U'}$ to $U$ belong to $\dmc(U)^{t_l\le 0}$ and to   $\dmc(U)^{t_l\ge 1}$, respectively. This is immediate from Lemma \ref{lcont}.

Again, we consider the closed embedding $i:Z\to S$ complementary to $j$. Now, the idea is that 
$t_l$ for $\dmcs$ can be glued from those for $\dmc(U)$ and $\dmc(Z)$. Though we only have $t_l$-decompositions in the latter category (by the inductive assumption), this is sufficient to construct the $t_l$-decomposition of $M$.
Indeed, by Proposition \ref{pglu}(I) there exists a distinguished triangle $A\to M\to B$ such that $j^*A\in \dmcu^{t_l\le 0}$ and $i^*A\in \dmc(Z)^{t_l\le 0}$ (resp. $j^*B\in \dmcu^{t_l\ge 1}$ and $i^!B\in \dmc(Z)^{t_l\ge 1}$). By Lemma \ref{lfumts}(7), this triangle yields the $t_l$-decomposition of $M$.

\begin{rema}\label{rmain}

1. Actually, 
we do not need a complete characterization of 
 $t_l$ 
 for the proof. We only need a pointwise characterization of $t_l$  (cf. Lemma \ref{lfumts}(7)) and  Lemma \ref{lcont}  for it. 
 
Also note here: if we have {\bf any} $t$-structures for $\dmcs$, $\dmck$, and $\dmc(U)$ for any $U$ such that all possible $j^*$ 
and $j_K^*[-d]$ are $t$-exact, then the statement of Lemma \ref{lcont} for these $t$-structures is fulfilled automatically. Indeed, Proposition \ref{pcisdeg}(\ref{icont}) implies that $j^*(M^{\tau\ge 1})$ (resp. $j^*(M^{\tau\le -1})$) vanishes for some $U$, since $j_K^*(M^{\tau\ge 1})$ (resp. $j_K^*(M^{\tau\le -1})$) does.

Still, the author does not know how to verify Lemma \ref{lcont} for the version of (the description of) the motivic $t$-structure (over fields) given by Proposition 4.5 of \cite{beilnmot}.

2. Lemma \ref{lfumts}(7) yields that $t_l$ does not depend on the choice 
of (a version of) $\hetl$ over $S$; see also Remark \ref{rind} 
below.

\end{rema}

Now we prove that it suffices to verify the conservativity of $\hetl$ and the existence of $t_l$ over universal domains.

 \begin{pr}\label{pred}
 1. Assume that $\hetl$  is conservative on $\dmck$ for all $K\in \sss$.  Then the same is true for $\dmcs$. 

2. Suppose that 
$\hetl$ is conservative  over some set of universal domains $K_i$ of certain characteristics $p_i\neq l$ (one of $p_i$ can be $0$). 

Then $\hetl$ is conservative over any (reasonable) $S$ that satisfies the following conditions: the characteristic of 
any point of $S$ is one of $p_i$. 

3. Suppose that 
$t_l$ exist over some universal domains $K_i$ of characteristics $p_i$. 

Then  $t_l$  also exists over any 
(reasonable) $S$ as in the previous assertion.

\end{pr}
\begin{proof}

1. Immediate from Proposition \ref{pcisdeg}(\ref{iconsp}) and Theorem \ref{tetre}(I).

2. The previous assertion yields: it suffices to verify that 
$\hetl$ is conservative  over any characteristic $p$  field $K$ (here $p$ could be $0$) if it is so over some universal domain $K'$ of characteristic $p$.
We should verify: if $\hetl(M)=0$ for an $M\in \obj \dmck$, then $M=0$. We fix some $M$.

First we note that any object (and morphism) in $\dmck$ is defined over some finitely generated subfield $F$ of $K$
see Proposition \ref{pcisdeg}(\ref{icont}). 
 Besides, for any extension of fields 
 the corresponding base change functor for 
 $D^b_cSh^{et}(-,\ql)$ is conservative  (see Theorem \ref{tetre}(III.1)). 
 Hence we obtain that $\hetl(M_F)=0$ for the corresponding $M_F\in \dmc(\spe F)$.
 Therefore, we may assume that $K\subset K'$;
denote the corresponding morphism by $b$. 
 We have $\hetl(b^*M)=0$; hence $b^*M=0$ and 
 the conservativity of $b^*$ (see Proposition \ref{pcisdeg}(\ref{iconspr})) yields the result.

3. Theorem \ref{tmts} implies: it suffices to verify that 
$t_l$ exists over any characteristic $p$  field $K$ (here $p$ could be $0$) if it exists over some universal domain $K'$ of characteristic $p$.
We prove 
this using an argument that is rather similar to the one above.

Again, in order to prove the existence of  
$t_l$ it 
suffices to verify the orthogonality axiom and the existence of $t_l$-decompositions for the classes  described in Definition \ref{dmts}(3).

Arguing as above, we obtain that it suffices to verify: if  
$t_l$ exists over $K'$, it also exists over  its 
subfield $K$.

First we consider  
 an algebraically closed $K\subset K'$. Our arguments along with Lemma \ref{lcont} yield:  if a $t_l$-decomposition of $Z_{K'}$ for a $Z\in \obj \dmck$ exists in $K'$, a $t_l$-decomposition of $Z_U[d]$ exists over 
some smooth connected 
$K$-variety $U$ of dimension $d$ (i.e. a $t_l$-decomposition 
$Z_1\to Z_U[-d] \to Z_2$ of $Z_U[-d]=u^*Z[-d]$ exists in $\dmc(U)$, for $u:U\to \spe K$ being the structure morphism of $U$). 
Similarly, if we have a non-zero $h\in \dmck(M,N)$, $M\in \dmck^{t_l\le 0}$,  $N\in \dmck^{t_l\ge 1}$, then it vanishes over a certain $U$ (since it vanishes over $K'$).

We denote by $s:\spe K\to U$ the embedding of some $K$-point of $U$ into $U$. 
Then $h=s^*u^*h$; hence $h=0$. Now, Lemma \ref{lfumts}(5) yields that $u^*[d]$ is $t$-exact. 
Since it is also conservative, we obtain: it suffices to verify that $Z_i=u^*s^*Z_i$ (for $i=1,2$). Since $\hetl$ is conservative by the previous assertion, it suffices to verify that $\hetl(Z_i)=(u\circ s)^*\hetl(Z_i)$. Now, it remains to note that $\hetl(Z_i)$ 
can be obtained by applying $u^*[d]$ to the $t_l$-decomposition of $\hetl(Z)$.


It remains to prove that 
$t_l$ exists over $K$ if it exists over its algebraic closure. Similarly to the reasoning above, we obtain: for any $Z\in \obj \dmck$    
there exists a finite extension $F/K$ (of some degree $d>0$) such that a $t_l$-decomposition of $Z_F$ exists (in $\dmc(F)$), and also any $h\in \dmck(M,N)$, $M\in \dmck^{t_l\le 0}$,  $N\in \dmck^{t_l\ge 1}$, vanishes over a certain $F$. 
Then Proposition \ref{pcisdeg}(\ref{itr}) yields 
the existence of $t_l$ over $K$. Indeed, if $f:F\to K$ is the corresponding morphism, loc. cit.  yields that the morphism $\dmck(M,N)\to \dmc(F)(f^*M,f^*N)$ is injective. Besides, since $f_*$ is $t$-exact (see Lemma \ref{lfumts}(4)), we obtain that a $t_l$-decomposition exists for $f_*f^*Z$; hence it exists for $Z$ 
also (since $\dmcs^{t_l\le 0}$ and  $\dmcs^{t_l\ge 0}$ are idempotent complete). 

\end{proof}

We also recall here Proposition \ref{pcisdeg}(\ref{ivoemot}); it yields that  $\dmc(K_i)$ is the category $\dmgm(K_i)$ of  Voevodsky's motives. Hence it suffices to verify   the conservativity of $\hetl$ and the existence of $t_l$ for the latter categories (cf. Remark \ref{rind}).

\begin{coro}\label{ctvs}
1. Suppose that 
$t_l$ exist over some universal domains $K_i$ of certain characteristics $p_i$. 
Then for any $S$ as in Proposition \ref{pred} Chow-weight filtrations and spectral sequences for $\hetlz$ over $S$ can be 
 lifted to $\mm(S)$.

2. Suppose that $t_l$ exists for $\dmck$ where $K$ is some  universal domain of characteristic $p$ ($p$ is either a prime or $0$); let $S$ be a very reasonable scheme of characteristic $p$. 
Then  there exists a weight filtration for $\mm(S)$ with the corresponding functors $\wmmli$ such that for any $i\in \z$ we have: $W_{i} 
\hetlz\cong \hetl\circ \wmmli\circ \htlz$.
\end{coro}

\begin{proof}
1. Immediate from 
Remark \ref{rwss}(2).

2. By Theorem \ref{tdeg}(II.1),  Chow-weight spectral sequences  degenerate (at $E_2$) for $H'=\hetlz$. Hence Proposition \ref{pdeg}(II) yields the degeneration of  Chow-spectral sequences
also for $H=\htlz$. Therefore part I of loc. cit. yields the existence of a weight filtration for $\mm(S)=\hrtl$ 
 such that $W_{i} 
\htlz\cong \wmmli\circ \htlz$. It remains to apply Proposition \ref{pwss}(I). 

\end{proof}

\begin{rema}\label{rintmot}
1. Certainly, the assumptions of the Corollary also yield that for any open embedding $j:U\to S$ one can lift 
$j_{!*}$ from $\shm$ to $\mm(-)$. In particular, if $U$ is regular and dense in $S$, 
$j_{!*}\q_U\in \dmcs^{t_l=0}$ could be called the 'intersection motif' of $S$; it corresponds to the $\ql$-adic \'etale intersection homology of $S$.

2. Conversely to part 2 of the Corollary, suppose that for some 
 $S$ there exists some weight filtration for $\mm(S)$ such that $H_i^{t_l}(\chows)$ is of weight $i$ (for any $i\in \z$; this assumption can be reduced to the following ones:  $\q(j)$ is of weight $-2j$ for any $j\in \z$, whereas $p_!$ respects weights in the corresponding sense for $p$ being a projective morphism of schemes). Then for $H=H^{t_l}_0$ one can easily see that $T(H,-)$ degenerates at $E_2$ (cf. the proof of Theorem \ref{tdeg}(II.1)). Certainly, this yields Conjecture \ref{cdeg} in this case  (see Proposition \ref{pdeg}(II.1)). We obtain a good reason to believe Conjecture \ref{cdeg} (for a general $S$).

3. Instead of assuming that $t_l$ exists over a universal domain $K$ (of characteristic $p$), it suffices to assume that it exists over all members of a family $K_i$ of fields such that any finitely generated $L$ of characteristic $p$ embeds into one of $K_i$. 
In particular,   one could take $K_{i}$  being algebraically closed fields of characteristic $p$ such that  their transcendence degree is not bounded (by any natural number).

\end{rema}

\subsection{Certain consequences of the existence of a nice motivic $t$-structure}\label{scons}

Now we derive certain 
consequences from the existence of a nice motivic $t$-structure for $\dmcs$; we will need some of them below in order to make a certain inductive step.

\begin{pr}\label{pcor}

Suppose that a nice $t_l$ exists over $S$; let $m\in \z$, $M\in \obj\dmcs$. Then the following statements are fulfilled.

I.1. The category $\mmm(S)=\mm(S)\cap \hwchow[m]$ is (abelian) semisimple.

2. If $M\in \dmcs^{t_l=0}$, then it  possesses an increasing filtration $W_{\le r,MM}M$, $r\in \z$, whose $j$-th factor belongs to $\mm_j(S)$ for any $j\in \z$; this filtration is $\hrtl$-functorial in $M$.

3.  $M\in \dmcs_{\wchow\le m}$ (resp. $M\in \dmcs_{\wchow\ge m}$) if and only if  for
 any $j\in \z$ we have $(W_{m+j}H_j^{t_l})(M)=H_j^{t_l}(M)$  (resp. $(W_{m+j-1}H_j^{t_l})(M)=0$).

4. $M\in \dmcs_{\wchow\le m}$ (resp. $M\in \dmcs_{\wchow\ge m}$) if and only if for any $j\in \z$ we have $(W_{m+j}\hetlj)(M) =\hetlj(M)$ (resp. $(W_{m+j-1}\hetlj)(M) =0$). 

5. If $M\in \obj \chows\ (\subset \obj \dmcs)$,   
then it can be decomposed into a direct sum of objects of $MM_j(S)[-j]$; this decomposition is unique up to a (non-unique) isomorphism.

6. If $M\in \obj \mmm(S)$, then it can be decomposed as a direct sum of simple objects of $\mmm(S)$; this decomposition is unique up to an isomorphism.


II Let $S$ be of finite type over $\sfp$ (for $p\neq l$); $M\in \dmcs^{t_l=0}$. Consider the weights for $\dhsl$ defined in \S5 of \cite{bbd} (cf.  Proposition \ref{phuwe}(I)). 
 Then
 $M\in \dmcs_{\wchow\le m}$ (resp. $M\in \dmcs_{\wchow\ge m}$) if and only if 
$\hetl(M)$ is of weight $\le m$ (resp. of weight $\ge m$). 
\end{pr}
\begin{proof}

I
(1--3,5): Immediate from Proposition \ref{ptrans}(II).

4. First we note that $\hetlj(M)\cong \hetlz(H^{t_l}_{j}(M))$. Applying assertion I3 we obtain: we may assume that $M\in \dmcs^{t_l=0}$ and consider only $m=0$. Then 
applying  Proposition \ref{pdeg}(II2) for $F=\hetlz$  we obtain the result (we also use Proposition \ref{pdeg}(II2) in order to relate the weight filtration for $\mm(S)$ with Chow-weight spectral sequences). 

6. Immediate from the semisimplicity of $\mmm(S)$.

II  Immediate  from assertion I5 along with Theorem \ref{tdeg}(I).

\end{proof}

\begin{rema}\label{rhaco}

1. $\mmm(S)$ could be called the category of {\it pure motives of weight $m$} (over $S$).

2. Consider a category $MS$ of 'homological $S$-motives' whose objects are $\dmcs_{\wchow=0}$, and $$MS(M,N)=\imm (\dmcs(M,N)\to \bigoplus_{m\in \z} 
\shs(\hetlm(M), \hetlm(N)). $$ We conjecture that it is (anti)-isomorphic to the category $M(S)$ described in Definition 5.9 of \cite{haco}
 (cf. 
Remark 2.1.2 of \cite{brelmot}).

We obtain: if 
a nice $t_l$ exists over $S$, then $MS$ is isomorphic to the direct sum of $\mmm(S)$ (as additive categories). 
Hence we obtain that $MS$ is semisimple (cf. Theorem 5.13 of \cite{haco}); so it could also be called the category of 'numerical motives'. It is also easily seen that for any $M\in \dmcs_{\wchow=0}$ the kernel of the projection $\mm(S)(Z,Z) \to MS(Z,Z)$ is a nilpotent ideal (cf. Theorem 6.9 of ibid.).

3. So, we proved that $\hetl$ 'strictly respects weights' if $S$ is of finite type over $\sfp$; this is also true for $\hetlw$ if $S$ is of finite type over $\sq$. In \cite{wildat} 
a similar statement was established unconditionally for {\it Artin-Tate} motives over number fields.

4. Assume that $I\in \obj \dhsl$ is semisimple
(i.e. that it is a direct sum of shifts of semisimple objects of $\shs$) 
 and that $I$ is a retract of $I'=\hetl(N)$ for some $N\in \obj \dmcs$ (under our assumptions, this is easily seen to be equivalent to 
$I$ being {\it semisimple of geometric origin} in the sense of \S6.2.4 of \cite{bbd}). Then our assumptions yield:
$I$ is a retract of $\bigoplus_{n\in \z} \hetlz (M_n)[n]$ for some  $M_n\in\obj \bigoplus_{m\in \z}\mmm(S)$.

Indeed, it suffices to verify this statement for a simple $I\in \obj \shs$; hence we can assume that $N\in \obj \mm(S)$.
Then we have morphisms $I\to \hetlz(W_{\le m,MM}N)$ for all $m\in \z$. Since $I$ is simple, these morphisms are either zero or embeddings; hence  $I$ is a retract of one of  $\hetlz (Gr_m^W(N))$.

5. Using the results of (\S1.2 of) \cite{btrans} (some of which were stated above) one can derive some more consequences from the existence of a nice $t_l$.

\end{rema}

\subsection{Reducing the existence of a nice 
$t_l$ to the universal domain case}\label{snice}

Now we are ready to prove our second main result.

\begin{theo}\label{tvmts}

Suppose that for any 
 point of a very reasonable scheme $S$ the category $\dmck$ 
possesses a nice $t_l$. 
Then the same is also true for $\dmcs$.

\end{theo}
\begin{proof} 
By Theorem \ref{tmts}, 
$t_l$ for $\dmcs$ exists. It remains to verify that $t_l$ is transversal to $\wchow$.
For any (fixed) $M\in \dmcs^{t_l= 0}$ and $m\in \z$ we should verify the existence of  
a 
nice decomposition of $M$ (see Definition \ref{dnwd}).

Again, we apply the Noetherian induction, and assume that the statement is fulfilled over any proper closed 
subscheme of $S$.

Let $K$ be a generic point of $S$. 
Since $j_K^*[-d]$ is  $t$-exact and $j_K^*$ is weight-exact, a nice choice of (\ref{ewd}) (with the corresponding $m_K=m-d$) 
 exists for $j_K^*M[-d]$ (in $\dmck$; see Theorem \ref{twchow}(III)).
By loc. cit. 
and Lemma \ref{lcont}, there exist an open 
embedding $j:U\to S$ ($U$ contains $K$) 
along with a nice choice 
\begin{equation}\label{efux}
A\stackrel{f'}{\to} j^*M\stackrel{g'}{\to} B
\end{equation}
of (\ref{ewd}). 

We verify that this choice can be lifted to a one for $M$.
 We apply $j_{!*}$ to (\ref{efux}). Since $j_{!*}$ preserves monomorphisms and epimorphisms (see Proposition \ref{pglu}(IV\ref{imonoepi}), we
 obtain a three-term complex as in (\ref{ecom}) (i.e. $f=j_{!*}f'$ is monomorphic, and $g=j_{!*}g'$ is epimorphic). The  
middle-term homology object $H_{mid}$ of the complex obtained belongs to $i_*\dmc(Z)$ by part IV\ref{icohom} of loc. cit. 
 Since $i_*$ is $t$- and weight-exact, the inductive assumption yields that a nice choice of (\ref{ewd}) exists
for $H_{mid}$. 
 Now suppose that
  $j_{!*}(W_{\ge m+1,MM}j^*M)\in \dmcs_{\wchow\ge m+1}$ and $j_{!*}(W_{\le m,MM}j^*M)\in \dmcs_{\wchow\le m}$. Then we can choose 'trivial' 
  nice decompositions for these objects;
hence Proposition \ref{ptrans}(I) would yield that a nice 
decomposition exists for $j_{!*}j^*M$. Now, applying
Proposition \ref{ptrans}(I) again along with Proposition \ref{pglu}(IV\ref{ibiext}) we obtain that a nice decomposition exists for $M$ also.

Hence it remains to verify that $j_{!*}$ maps $\dmcu^{t_l=0}\cap 
\dmcu_{\wchow\ge m+1}$ into $\dmcs_{\wchow\ge m+1}$, and maps $\dmcu^{t_l=0}\cap 
\dmcu_{\wchow\le m}$ into $\dmcs_{\wchow\le m}$. 

We fix some $M\in \dmcu^{t_l=0}\cap\dmcu_{\wchow\ge m+1}$ (resp. $M\in \dmcu^{t_l=0}\cap \dmcu_{\wchow\le m}$). Since $j^*j_{!*}M\cong M$, it suffices to verify that $i^! j_{!*}M\in \dmc(Z)_{\wchow\ge m+1}$  (resp. $i^* j_{!*}M\in \dmc(Z)_{\wchow\le m}$). 

The inductive assumption for $Z$ reduces the latter fact to 
a certain calculation of weight filtrations for $\hetln$ of the corresponding motives; see Proposition \ref{pcor}(I4). In this form the statement follows immediately from Proposition \ref{pglu}(\ref{iim}) and Theorem \ref{tdeg}(II2) (along with Theorem \ref{tetre}).

\end{proof}

\begin{rema}\label{rcompat}

1. Our arguments demonstrate that the notions of weight structure and of its transversality with $t$-structures are really important 
for the study of the 'weight filtration' of $\dmcs^{t_l=0}$ (cf. \S\ref{sconj} below). Indeed, it  seems that one cannot apply our gluing argument in the setting of filtered abelian categories (though  possibly one could find a way to apply some of the corresponding arguments of \cite{haco} in our context). 

2. In contrast to the setting of the Theorem \ref{tmts}, we cannot prove the niceness of $t$ when $S$ is not very reasonable (without assuming Conjecture \ref{cdeg} for it). The problem is that the 'weight-exactness' of $j_{!*}$ does not follow from the (Noetherian) inductive assumption considered in the proof of the theorem. Indeed, let $S=\spe \z_{(p)}$ (for a prime $p\neq l$); then one can glue $t_l(\spe \mathbb{F}_p)$ with any 'shift' of $t_l(\spe \q)$ (here we assume that $t_l(\spe \mathbb{F}_p)$ and $t_l(\spe \q)$ exist, and consider $(\dmc(\spe \q)^{t'\le 0},\dmc(\spe \q)^{t'\ge 0})= (\dmc(\spe \q)^{t_l\le i},\dmc(\spe \q)^{t_l\ge i})$ for any $i\in \z\setminus \ns$). Then the niceness of $t_l$ over $\spe \q$ is equivalent to the niceness of $t'$; yet it seems highly improbable for  $j_{!*}$ to be weight-exact for the weight structure obtained via this 'shifted gluing'. Hence in order to control the niceness of $t_l$ for $S$ in this case, one needs some 
'extra' information on it. It seems quite reasonable to control motives via their homology; to this end we have  
to extend Theorem \ref{tdeg}(II) to this case (cf. some alternative arguments in \S7 of \cite{wildiat}). 

\end{rema}

Now we prove that it suffices to verify the niceness of $t_l$ over universal domains (only).

\begin{pr}\label{pcmtvs}
1. Suppose that a nice 
$t_l$ exists 
over a universal domain $K$ of  characteristic $p>0$. 
Then 
a nice $t_l$ exists over any very reasonable $S/\sfp$. 

2. Suppose that  $t_l$ 
exists over the field of complex numbers. Then a nice $t_l$ exists over any 
very reasonable 
 $\sq$-scheme. 

\end{pr}

\begin{proof} 
1. The proof is rather similar to 
that of Corollary \ref{ctvs} (along with Proposition \ref{pred}).  
We will only 
sketch it outlining the difference.

Again, it suffices to verify: if the transversality property is fulfilled for motives over a field $L$, then it is fulfilled over any its algebraically closed subfield, and over its subfield $K$ such that the extension $L/K$ is algebraic. 

Both of these statements can be proved using the arguments  in the proof of loc. cit. 
Indeed, by Proposition \ref{ptrans}(III), we should verify that
for any $M\in \obj \dmck$ we have $(W_{m}H_0^{t_l})(M)\in \dmck_{\wchow\le m}$ and $M/(W_{m}H_0^{t_l})(M)\in\dmck_{\wchow\ge m+1}$  
(for all $m\in \z$).
This can be easily done by combining the  arguments from the proof of Corollary \ref{ctvs}(1) with 
Theorem \ref{twchow}(VI) (see also Remark \ref{reas}(I.1));
note that $f_*=f_!$ is weight-exact if $f$ is a finite morphism.

2. The statement is immediate from the previous assertion along with Proposition 1.5 of \cite{beiln}. 
\end{proof}

\begin{rema}\label{rnicesub}
1. It is also easily seen that if $t_l$ is nice over $S$, it is also nice over all of its subschemes and residue fields. Indeed, it suffices to note that for any open immersion $i$ and (the complementary) closed embedding $j$ the functors $i_*$ and $j^*$ are exact with respect to $t_l$ and $\wchow$, whereas $i_*$ is a full embedding, $j^*$ is a localization functor, and $\imm i_*=\ke j^*$.

Certainly, this observation is far from being very exciting; yet it will 
make some of the formulations in \S\ref{sdecomp} nicer.
 
 2. Remark 2.1.4 of \cite{btrans} describes a funny way to produce new examples of transversal weight and $t$-structures (out of 'old' ones for a triangulated $\cu$). To this end one should consider the so-called 'truncated categories' $\cu_N$ (that are 'usually' defined for all $N\ge 0$). 
 For our $t,w,\ \cu=\dmcs$, we have $\cu_0=K^b(\chow(S))$. So (if certain 'standard' conjectures as listed in \S\ref{sconj} below hold) this category shares several nice properties with $\dmcs$;  
this statement does not seem to be obvious.

\end{rema}

\section{Supplements}\label{ssupl}

In \S\ref{sconj} we verify that the existence of $t_l$ and  
its niceness (over a very reasonable scheme $S$) follow from certain (more or less) 'standard' motivic conjectures (over algebraically closed fields;
 here we use certain lists of those taken from \S1 of \cite{beilnmot} and \S2 of \cite{ha3}). 
 
 In \S\ref{sdecomp} we note that our results yield a  certain 'motivic Decomposition Theorem' (modulo the conjectures mentioned). 
In particular, we characterize pure motives over $S$ in terms of those over its residue fields. This enables us to calculate $K_0(\dmcs)$.

In \S\ref{sindl} we  extend (somehow) our results  from the case of $\szol$-schemes to the case of $\spe\z$-ones, 
 and  prove that the $t$-structure obtained does not depend on the choice of the corresponding $l$'s. Here we need to assume that  the numerical equivalence of cycles is equivalent to $\qlp$-adic homological one (for any  $l'\in \p$ and over universal domains of characteristic $\neq l',0$).

\subsection{Relating 
the existence of a (nice)  $t_l$ with  'standard' motivic conjectures} 
\label{sconj}

First we address the question: which (more or less) 'classical' motivic conjectures ensure the existence of 
$t_l$ over $S$, that is nice if $S$ is a very reasonable scheme. By  virtue of the results above, to this end it suffices to treat motives over universal domains only. 
So we consider motives over a universal domain $K$ of characteristic $p\neq l$ ($p$ is either a prime or $0$); recall that $\dmc(K)\cong \dmgm(K)$. None of the results of this paragraph are essentially original (unless we combine them with some  of other our results).

\begin{pr}\label{pbeil}
The existence of a nice $t_l$ for $\dmgm(K)$ is equivalent to (the conjunction of widely believed to be true) conjectures A--C of \S1.2 of \cite{beilnmot}.
\end{pr}
\begin{proof}
Conjectures A and B of loc. cit. state that $\ql$-adic \'etale cohomology on
Voevodsky's $\dmge(K)$
is strictly compatible with a certain $t$-structure  for it.
 Since \'etale cohomology is compatible with Tate's twists, this $t$-structure can be extended to the whole $\dmgm(K)=\dmck$; we will denote this extension by 
  $t_{MM}$.
 
 Now,
it is easily seen that $t_{MM}=t_l$. Indeed, composing the \'etale cohomology  with Poincare duality for $\dmck$ 
 one obtains (a certain version of) \'etale homology for it. Note here that the Poincare duality for $\dmgm(K)$  exists for $K$ of any characteristic (by an  argument of M. Levine described in Appendix B of \cite{hubka}). 

Next, Conjecture C of \cite{beilnmot} states that the $t_{MM}$-homology objects for motives of smooth projective varieties (over $K$) with respect to $t_l$ are semisimple in 
$\hrtl$. Then the same assertion is true for arbitrary Chow motives ('in our sense'; here we apply the compatibility of $\hetl$ with Tate twists). Now,  Proposition 1.4(ii) of \cite{beiln}
yields the existence of the corresponding Chow-Kunneth decompositions (see Remark \ref{rchk});
hence the functors $H^{t_l}_i[-i]$ 
 respect  Chow motives. 
 Thus  $t_l$ is nice (over $K$) by Proposition \ref{ptrans}(V).

The converse implication is even easier (and is not really interesting for us).

\end{proof}

\begin{rema}\label{rind}
1. Here and 
 throughout this paper we use the following  observation: though the author doesn't know whether all  possible versions of the ($\ql$-) \'etale homology realization for motives over a field $K$ are isomorphic,
one can still be sure that all of them yield the same $t_l$. Indeed, we have spectral sequences $T_{\wchow}(\hetlz,-)$ (for any version of $\hetlz$) that degenerate at $E_2$ (in this case; see Theorem \ref{tdeg}(II.1)). Hence  $\hetlm(M)$ vanishes (for some $m\in \z$) if and only if $E_2^{p,m-p}(T)=0$ for any $p\in \z$. Now, in order to calculate $E_2^{p,m-p}(T)$ it suffices to know the restriction of $\hetl$ to $\chow(K)$ (see Proposition \ref{pwss}(I)), and certainly the latter does not depend on the choice of the version for $\hetl$.

2. Combining Proposition \ref{pred}(2) with this spectral sequence argument, we obtain that the conservativity of the \'etale realization of motives (over fields or over general $\zol$-base schemes) follows from the following conjecture: if $K$ is an algebraically closed field of characteristic distinct from $l$, then any morphism of $\chow(K)$-motives that yields an injection on their \'etale (co)homology, splits (cf. Proposition 7.4.2 of \cite{mymot}). Note that the latter conjecture easily follows from the niceness of $t_l$ over $K$ (since $t_l$ 'splits' Chow motives into direct sums of objects of semisimple categories $\mmm(S)[-m]$, whereas $\hetl$ is conservative on $\mmm(S)[-m]$). Now, by the virtue of the results below, the niceness of $t_l$ follows from   'standard' conjectures. Hence, there is a good reason to believe this ('Chow-splitting') conjecture.

Respectively, it could be interesting to study the connection of our results with those of \cite{ayocons}.
\end{rema}

Now we 
verify (briefly) the following statement: the existence of a nice $t_l$ 
over $K$  follows
from the conjectures stated in \S2 of \cite{ha3}. 
Being more precise, one needs the standard conjecture D of loc. cit. (that $\ql$-homological equivalence of cycles 
coincides with numerical one), and Murre's conjectures 
A,D, and Van.

First we note that Murre's conjecture A yields the 
existence of Chow-Kunneth decompositions of motives of  smooth projective  varieties (over $K$) i.e. any  such motif  can be decomposed (in $\chow(K)$) into a direct sum of  motives each of those has only one non-zero $\ql$-adic (co)homology group. Here and below we can consider $\hetl$ instead of \'etale cohomology; cf. the proof of Proposition \ref{pbeil}. 
 Next, (the proof of) Proposition 2.4 of \cite{ha3} implies that the conjectures mentioned imply all the remaining Murre's conjectures (we can apply loc. cit. here since the 
Lefschetz type standard conjecture B used in its proof follows from standard conjecture D by the main result of \cite{smi}).  We define $\mmm(K)$ as the subcategory of $\chow[m]\subset \dmgm(K)$ consisting of objects whose $\ql$-\'etale (co)homology is concentrated in degree $0$.
 
 Proposition 2.3 of 
\cite{ha3} yields that the categories $\mmm(K)$ are isomorphic to the corresponding pieces of the category of $\ql$-\'etale homological motives. Conjecture D embeds them into the category of numerical motives (which is semisimple by the main result of \cite{ja}); hence they are semisimple also. 
  Next, the arguments used for the proof of Proposition 2.9 of \cite{ha3} 
 yield for $\mmm(K)$ the orthogonality 
 conditions of Proposition \ref{ptrans}(IV). 
 Besides, by Lemma 1.1.1(6) of \cite{btrans} these conditions also yield that the category
 $C\subset \chow(K)$ with
 $\obj C=\bigoplus \obj \mmm(K)[-m]$ is idempotent complete; hence $C=\chow(K)$.  Since $\lan \chow(K)\ra =\dmgm(K)$ (see Proposition \ref{pcisdeg}(\ref{ivoemot})), Proposition \ref{ptrans}(IV) yields that $\dmgm(K)$ possesses a  $t$-structure $t_{MM}$ that is transversal to $\wchow$. Since all objects of $\hrt_{MM}$ possess filtrations whose factors belong to $\mmm(K)$ (see Proposition \ref{ptrans}(II4)), we obtain that 
 $\obj \hrt_{MM}\subset \dmc(K)^{t_l=0}$; hence $\hetl$ is $t$-exact with respect to $t_{MM}$. Moreover, Murre's conjecture Van yields that $\hetl$ does not kill non-zero objects of $\mmm(K)$ (since it does not kill non-zero Chow motives). 
 Then Proposition \ref{ptrans}(II4) also implies that  $\obj \hrt_{MM}= \dmc(K)^{t_l=0}$; cf. the proof of Proposition \ref{pcor}(I4). Hence $t_{MM}=t_l$ (see Remark \ref{rts}(3)), and we obtain the result desired.

\begin{rema}\label{rha3}

1. Alternatively,   one 
can prove the existence of the motivic $t$-structure for $\dmgm(K)$
using
the  arguments from the  proof of Theorem 3.4 of \cite{ha3}, whereas (the proof of) Proposition 2.8 of ibid. allows us to verify the conditions of Proposition \ref{ptrans}(V) that ensure (in this case) that $\wchow(K)$ is transversal to $t_l$.

2. For a characteristic $0$ field $K$ one can apply the results of Corti and Hanamura directly (after replacing \'etale cohomology by 
$\hetl$ using Poincare duality). Indeed, it was proved in \S4 of \cite{mymot} that in this case Hanamura's triangulated category of motives is isomorphic to $\dmgm(K)^{op}$.




3. It  seems that the existence of 
$t_l$ without any additional assumptions does not imply its niceness  (at least, easily) over positive characteristic fields (one also needs to assume the Hodge standard conjecture or the conjectures mentioned above  for this matter). 
Over (very reasonable $\sq$-schemes and) characteristic $0$ fields one does not need any extra assumptions; cf. Proposition \ref{pcmtvs}(2) (and also Proposition 2.2 of \cite{beiln}).
\end{rema}

\subsection{Our 'motivic Decomposition Theorem'}\label{sdecomp}

Our results easily yield a motivic version of the celebrated Topological Decomposition Theorem (for perverse sheaves; see Remark \ref{rdect}). In particular, we characterize pure motives (see Remark \ref{rhaco}(1)) 
'pointwisely'.
In order to formulate our results, 
we need a certain intermediate image functor for $j_K$, where $K\in \sss$.

First let $K$ be a generic point of $S$ of dimension $d$ (see \S\ref{scheme}). Suppose that a nice $t_l$ exists over $S$; then it also exists over $K$ (see Remark \ref{rnicesub}(1)). We define $j^d_{K!*}$ for $M\in \mmm(K)$ ($m\in \z$) in the following way. First we lift $M[d]$ to a certain $M_U\in \mm_{m+d}(U)$ for some open $U\subset S$, $K\in U$; here we use Lemma \ref{lcont} and Theorem \ref{twchow}(III) (cf. the proof of Theorem \ref{tvmts}). $M_U$ is semisimple in $\mm(U)$ (see Proposition \ref{pcor}(I6));
and we take $M'_U$ being the sum of  those components of $M_U$ that are not killed by $j_K^{U*}$ (for the corresponding morphism $j_K^{U}:K\to U$; note that $M'_U$ is  determined by $M_U$ uniquely up to an isomorphism). Lastly, we set   $j^d_{K!*}M=j_{!}M'_U$, where $j:U\to S$ is the corresponding open immersion.

Now, let $K$ be an arbitrary point of $S$ of dimension $d$, whose closure is $Z\subset S$; $i:Z\to S$ is the corresponding embedding. Then we lift $M[d]$ to $\mm_{m+d}(Z)$ using the procedure described above, and then apply $i_*$ in order to obtain $j^d_{K!*}M$. Certainly, here we use $t_l-$ and $\wchow$-exactness of $i_*=i_!$. Besides, note: if we denote the composite immersion $U\to Z\to S$ by $j$, then we would have 
\begin{equation}\label{eiimk}
j^d_{K!*}M=\imm(H_0^{t_l}j_!M'_U\to H_0^{t_l}j_*M'_U) \end{equation} in this case (also). 

\begin{lem}\label{liim}
1. $j^d_{K!*}M$ does not depend on any choices (if a nice $t_l$ exists over $S$). Moreover, $j^d_{K!*}$ 
yields a full embedding (of categories).

2. $j^d_{K!*}M$ is functorially characterized by the following condition:  it is a semisimple 
lift of $M$ to $\mm_{m+d}(S)$  none of whose direct summands are killed by $j_K^*$. 
\end{lem}
\begin{proof} This is an easy consequence of Proposition \ref{pglu}(IV\ref{issm}-\ref{isss}). 
\end{proof}

\begin{rema}
Alternatively, one could try to apply here the (somewhat parallel)  arguments of \S5 of 
\cite{scholfcohom}. 
 Yet some  adjustments (along with certain results of \S2.3 of \cite{brelmot}) are certainly needed to do this.
\end{rema}

\begin{pr}\label{pdecomp}

Assume that  a nice $t_l$ exists over $S$. Then for any $m\in \z$ any object of $\mmm(S)$ can be decomposed as a direct sum of $j^{d_i}_{K_i!*}M_i$ for $K_i\in \sss$ being of dimension $d_i$, and $M_i\in \mm_{m-d_i}(K_i)$ 
 being indecomposable objects. This decomposition is unique up to an isomorphism. Moreover, $M_i\cong H_{-d_i}^{t_l}(j_{K_i}^*M)$, whereas $K_i$ can be characterized by the condition that $H_{-d_i}^{t_l}(j_{K_i}^*M)\neq 0$.

\end{pr}

\begin{proof} First we verify that $M$ can be decomposed into a direct sum of some $j^{d_i}_{K_i!*}M_i$  (somehow). Since $\mmm(S)$ is semisimple, it suffices to prove: if $M$ is indecomposable, then it can be presented  as $j^{d_M}_{K_M!*}M_K$ for some $K_M\in \sss$  of dimension $d_M$ and $M_K\in \mm_{m-d}(K_M)$. We prove this statement by Noetherian induction (applying Remark \ref{rnicesub}(1) again). 

We take $K$ being a generic point of $S$. If $j_K^*M\neq 0$, 
Lemma \ref{liim} 
immediately implies that we can take $K_M=K$, $M_K=j_K^*M[-d]$. Conversely, if $j_K^*M=0$, then 
there exists an open immersion $j:U\to S$ ($K\in U$) such that $j^*M=0$.
 Hence for the complementary  closed embedding $i:Z\to S$ there exists a (simple) $M_Z\in \mmm(Z)$ such that $M\cong i_*M_Z$ (since $i_*$ is Chow-weight and $t$-exact). 
Hence it suffices to apply the inductive assumption to $M_Z$ (see (\ref{eiimk})). 

It remains to verify: for $K,K'\in \sss$ of dimensions $d$ and $d'$ respectively, $M\in \mm_n(K)$ ($n\in \z$) we have:  $H_{-d'}^{t_l}(j_{K'}^*j^d_{K!*}(M))=0$ if $K'\neq K$ and $=M$ otherwise. We consider three cases here: 1) $K'=K$, 2) $K'$ belongs to the closure $Z$ of $K$ in $S$, and 3) $K'$ does not belong to $Z$.

Denote the embedding of $Z$ into $S$ by $i$; denote the complementary immersion by $j$ and the morphism $K\to Z$ by $j_K^Z$. In case 1) it suffices to note that  $j^{Z,d}_{K!*}M$ is a lift of $M[d]$ to $\dmc(Z)$, whereas $i^*i_*=1_{\dmc(Z)}$. In case 3) it suffices to note that $j_{K'}^*$ factors through $j$ and that $j^*i_*=0$. In case 2) we can assume that $Z=S$ (since $i^*i_*=1_{\dmc(Z)}$); then our claim easily follows from Proposition \ref{pglu}(IV\ref{iim}).

\end{proof}

\begin{rema}\label{rdect}
1. The 'usual' Topological Decomposition theorem (see Theorem 5.7 of \cite{haco}) states (for $S$ being a variety over a field): if $X\to S$ is a proper morphism, $X$ is regular, then $f_*\ql{}_X\in \obj \dhsl$ splits as a direct sum of its $t$-homology, whereas its homology (perverse) sheaves can be presented as  direct sums of intermediate images of pure $\ql$-local systems supported on some subvarieties of $S$. 
We verify that this decomposition can be lifted to $\dmcs$
(hence, we can improve Theorem 5.14 of \cite{haco}) even if we replace $\ql{}_X$ here by  
$K=\hetl(N)$ for an $N\in \hetlz(\mm_n(X))$ (for some $n\in \z$) and do not demand $X$ to be regular.

First we note that $f_*K=\hetl (f_*N)$ (see Theorem  \ref{tetre}(I)), whereas $f_*N\in \dmcs_{\wchow=n}$ (see Theorem \ref{twchow}(II.1)). By Proposition \ref{pcor}(I5) we obtain that $t_l$ splits $f_*N$ into a direct sum of  objects of $MM_j(S)[n-j]$ (this statement is the relative generalization of the existence of Chow-Kunneth decompositions).

Hence in order to fulfill our goal it suffices  to verify (by the virtue of Proposition \ref{pdecomp}; for an $m\in \z$) for any $M\in \obj \mmm(S)$ that
 the (perverse) homology of the corresponding $j^{d_i}_{K_i!*}M_i$ can be presented as the 
the intermediate image of  a $\ql$-local system supported on some regular connected subvariety $U_i$ of $S$, whereas $K_i$ is the generic point of $U_i$ (cf. Theorem 4.3.1(ii) of \cite{bbd}). 
As we have verified above, for any such $U_i$ the motif  $j^{d_i}_{K_i!*}M_i$ can be presented as $ j_{U_i!*}M_{U_i}$ for some $M_{U_i}\in \obj \mmm(U_i)$  
 (here we denote by $j_{U_i!*}$ the composite of the intermediate image functor for the embedding of $U_i$ into its closure $Z_i$ with the direct image $\dmc(Z_i)\to \dmcs$). We should prove that we can choose $U_i,M_{U_i}$ such that   $\hetl(M_{U_i})\in \obj \shs$ is a local $\ql$-system on $U_i$.

By Theorem \ref{twchow}(III2), we can assume (if we choose $U_i$ to be small enough) that there 
exist: a regular scheme $U'_i$, a   
finite universal homeomorphism $g:U'_i\to U_i$, a smooth projective morphism $h:P_i\to U'_i$, and an $s\in \z$ such that $M_{U_i}$ is a retract of $(g\circ h)_*\q_P(s)[2s+m]$. It remains to note that the homology sheaves of $(g\circ h)_*\ql{}_{P_i}(s)\in \obj \dhuli$ are pure local systems (since this is true for the 'canonical' homology of this derived category, and $U_i$ is regular, we obtain that the perverse homology equals the canonical one).

2. More generally, consider $I$ being a retract of an object $I'$ of $ \hetl(\bigoplus_{m,j\in\z}\mmm(X)[j])$.
As noted in Remark \ref{rhaco}(4), this condition is fulfilled (in particular) if  $I$ is a semisimple 
 complex 
 of geometric origin (in the sense of \cite{bbd}). Then $f_*I'$ belongs to $ \hetl(\bigoplus_{m,j\in\z}\mmm(S)[j])$.
 Since local systems over subschemes of $S$ yield Krull-Schmidt subcategories, we obtain that $f_*I$ can be presented as a direct sum of retracts of $\hetl( j_{U_i!*}M_{U_i}[s_i])$ for $(U_i,M_{U_i},s_i)$ corresponding to $f_*I'$.

Thus we obtain a certain motivic analogue of Theorem 6.2.5 of \cite{bbd}.
Yet note that (in contrast with loc. cit.) $f_*$ does not preserve semisimplicity of perverse sheaves in general. Indeed, even if we take $S=\spe K$ for a (general) field $K$ and a smooth projective $X/K$, then the \'etale cohomology of $X_{K^{sep}}$ need not be semisimple as  $\operatorname{Gal}(K)$-representations (for example, for $K=\q_p$ we do not have semisimplicity for 
 $H^1_{et}$ of an elliptic curve with split multiplicative reduction; see Exercise 5.13 of \cite{siladv}). 


\end{rema}

Now we are (also) able to calculate $K_0(\dmcs)$.

\begin{coro}\label{ckz}
Define  $K_0(\dmcs)$ as the abelian group whose generators are $[C],\ C\in \obj \dmcs$; if
$D\to B\to C\to D[1]$ is a distinguished triangle in $\dmcs$ then we set
$[B]=[C]+[D]$.

Assume that  a nice $t_l$ exists over $S$.
Then $K_0(\dmcs)$ is a free abelian group with a basis 
indexed by isomorphism classes of indecomposable objects of $\mmm(K_i)$ for $K_i$ running through all elements of $\sss$, $m$ running through all integers.

\end{coro}

\begin{proof}
By Proposition 1.2.6 of \cite{btrans},  $K_0(\dmcs)$ is a free abelian group with a basis 
indexed by isomorphism classes of indecomposable objects of $\mmm(S)$ (for $m$ running through all integers). Now the result follows from Proposition \ref{pdecomp} easily. 
\end{proof}

\subsection{Changing $l$; the case of 
\texorpdfstring{$\operatorname{Spec}\,\z$}{Spec Z}-schemes}\label{sindl}

First we study the question when $\tlp$ exists and coincides with $t_l$ for prime $l\neq l'$ (we fix the primes, and define $\tlp$ similarly to $t_l$).  

\begin{pr}\label{pcttl} 
1. Suppose that $\hetl$ and $\hetlp$ exist and coincide over universal domains of all 
characteristics $\neq l,l'$. Then they also exist and coincide over any reasonable $\zollp$-scheme
$S$.

In particular, this assertion holds if $\hetl$ and $\hetlp$ are nice over the universal domains mentioned.

2. Suppose that for any prime $p\neq l,l'$ there exists a universal domain $K$ of characteristic $p$ such that: there exists a nice $t_l$ for $\dmck$, and $\hetlp$-homological equivalence of cycles coincides with the numerical equivalence one over $K$.  

Then $t_l$ and $\tlp$ exist and coincide on $\dmcs$ for any reasonable $\zollp$-scheme
$S$. 

\end{pr}
\begin{proof}

1. First we note that the niceness of $t_l$ and $t_l'$ for $\dmck$ yields that  they coincide (on $\dmck$) by Proposition 4.5 of \cite{beilnmot}. 

So, it remains to 
verify that $t_l$ and $t_l'$ exist and coincide over $S$ if this is true over universal domains (of all 
characteristics $\neq l,l'$).

By 
Corollary \ref{ctvs}(1), $t_l$ and $\tlp$  
 exist for $\dmcs$. By Lemma \ref{lfumts}(7), it suffices to verify that $\dmc(F)^{\tlp\le 0}=\dmc(F)^{t_l\le 0}$ and $\dmc(F)^{\tlp\ge 0}=\dmc(F)^{t_l\ge 0}$ for any point $F$ of $S$. An argument similar to the one used in the proof of  Corollary \ref{ctvs}(1) yields that we can replace $F$ by one of our universal domains (of the same characteristic).

2. By the previous assertion, it suffices to verify that $t_l$ coincides with $\tlp$ over $K$. 

Since all $\mmi(K)=\dmck^{t_l=0}\cap \dmck_{\wchow =i}$ are semisimple, we obtain that $\hetl$-homological equivalence of cycles 
is equivalent to numerical equivalence (over $K$).   Indeed, 
consider the functor  $\bigoplus Gr^{t_l}_m:\chow(K)\to \bigoplus \mmm(K)$  that 
is given by the direct sum of all (shifted) $t_l$-homology of Chow motives; see Remark \ref{rhaco}(2). It
kills exactly those morphisms of Chow motives that are $\ql$-homologically equivalent to zero as cycles (since $\hetl$ does not kill non-zero morphisms in $\mmm(K)$). 
It remains to note that $\ql$-homological equivalence is finer than the numerical one, and the category of numerical motives (over $K$) is semisimple (see
\cite{ja}). Besides, $\hetlp$ is also conservative on $\mmm(K)$ for any $m\in\z$.

Now, 
for a smooth projective $P/K$ 
Proposition 5.4 of \cite{klei} implies (in our setting) that
the (Chow)-Kunneth decomposition of the motif of $P$ 
 coming from $t_l$ is also its Kunneth decomposition with respect to $\hetlp$. Indeed, by loc. cit. the numerical equivalence classes of the corresponding projectors do not depend on the choice of a Weil (co)homology theory; note that (by the main result of \cite{smi}) we can replace the  Hodge standard conjecture by the Standard Conjecture D (see \S\ref{sconj}) in the assumptions of  \cite{klei}. 
Hence $\hetlp$ sends $\mmm(K)$ 
into 
$\dhklp^{t=0}$. Then Proposition \ref{pcor}(I2) 
easily yields
  that $\dmck^{\tlp= 0}=\dmck^{t_l=0}$; boundedness yields that $\tlp$ coincides with $t_l$ (see Remark \ref{rts}(3)). 

\end{proof}

Now suppose that $S$ is not (necessarily) a $\szol$-scheme. We note: $\hetl$ kills all motives 'supported on' $S\times_{\spe \z}\spe \fp$. In order to overcome this difficulty we introduce the following definition.

\begin{defi}\label{dmtsllp}
Let $l\neq l'$ be primes. Consider the  class $\dmcs^{\tllp\le 0}$ (resp. $\dmcs^{\tllp\ge 0}$) consisting of those $M\in \obj\dmcs$
that satisfy:  $\hetl(M)\in \dshsl^{t_l\le 0}$ and  $\hetlp(M)\in \dshslp^{\tlp\le 0}$ (resp. $\hetl(M)\in \dshsl^{t_l\ge 0}$ and  $\hetlp(M)\in \dshslp^{\tlp\ge 0}$).

If $(\dmcs^{\tllp\le 0},\ \dmcs^{\tllp\ge 0})$ yield a 
$t$-structure for $\dmcs$, we will say that (the $t$-structure) $\tllp$ exists over $S$. 

\end{defi} 

Now we observe that 'standard' conjectures imply: $\tllp$ exists over any (reasonable) $S$ and does not depend on $l$. We 
formulate a certain concise  
result of this kind here; some more (somewhat stronger) statements of this sort can also be easily proved.

\begin{pr}\label{pmts}
Let $l,l'$ be as above.

1.  Suppose that a nice $t_l$ exists over any universal domain of any characteristic $\neq l$; a nice $\tlp$ exists over a universal domain of characteristic $l$, and that  $\hetlp$-homological equivalence of cycles 
is equivalent to numerical equivalence  over any universal domain of any characteristic $\neq l',0$. Then $\tllp$ is a 
 $t$-structure over any (reasonable) $S$.

2. Suppose moreover that for any prime $p$ (distinct from $l,l'$) we have: $\hetp$-homological equivalence of cycles coincides with the numerical equivalence one over any universal domain of any characteristic $\neq p$.
 Then  $\tllp$ does not depend on the choice of the pair $l,l'$. 

\end{pr}
\begin{proof}
1. It is easily seen that $\tllp$ can be characterized similarly to Lemma \ref{lfumts}(7) i.e.: $M\in \dmcs^{\tllp\le 0}$ (resp. $M\in \dmcs^{\tllp\ge 0}$) if and only if
for any $K\in \sss$, $K$ is of 
dimension $d$, we have $j_{K}^*M[-d]\in \dmck^{\tllp\le 0}$ (resp. $j_{K}^!M[-d]\in \dmck^{\tllp\ge 0}$). Indeed, this is an easy consequence of loc. cit.

Next we note that over characteristic $0$ universal domains the homological equivalence of cycles relation does not depend on the choice of $l$ since it can be described in terms of singular (co)homology. 
Hence   Proposition \ref{pcttl} yields that $t_l=t_l'$ over any field of characteristic $\neq l,l'$.

It easily follows that $\tllp(S)$ can be glued from $t_l$ for $S[1/l]$ and $t_{l'}$ for $S\times_{\spe \z}\spe \fl$; see Proposition \ref{pglu}(II.1) and Proposition \ref{pcisdeg}(\ref{iglu}). 


2. It suffices to apply Proposition \ref{pcttl} (again) and the arguments described above. 

\end{proof}

\end{document}